\documentclass{article}
\usepackage[utf8]{inputenc}
\usepackage{amsmath,amssymb,amsthm,amsfonts,latexsym,graphicx,mathtools,enumitem}
 
\usepackage{url}
\usepackage{xcolor}
\usepackage{svg}
\usepackage{caption}
\usepackage{subcaption}
\usepackage{svg}
\usepackage{amsmath,mathrsfs}

\usepackage{hyperref}
\usepackage[capitalise]{cleveref}
\usepackage{tikz}

\theoremstyle{plain}
\newtheorem{theorem}{Theorem}
\newtheorem{lemma}[theorem]{Lemma}
\newtheorem{proposition}[theorem]{Proposition}
\newtheorem{corollary}[theorem]{Corollary}
\newtheorem{conjecture}{Conjecture}

\theoremstyle{definition}
\newtheorem{definition}[theorem]{Definition}
\newtheorem{example}[theorem]{Example}
\theoremstyle{remark}
\newtheorem{remark}[theorem]{Remark}
\DeclareMathOperator{\im}{\mathrm im}

\newcommand{\lk}{{\rm lk}}

\usepackage{stmaryrd} 
\usepackage[affil-it]{authblk}
\usepackage{booktabs}   
\usepackage[numbers,sort&compress]{natbib}

\usepackage[margin=3.5cm]{geometry}

\date{}

\allowdisplaybreaks

\title{Extremal Betti Numbers and Persistence in Flag Complexes}

\newcommand{\Ind}[0]{{\rm Ind}}
\newcommand{\BF}[0]{\beta^{{\rm FL}}}
\newcommand{\BT}[0]{T\beta^{{\rm FL}}}
\newcommand{\BC}[0]{\mathcal{B}^{{\rm FL}}}

\newcommand{\stt}{{\rm st}}

\newcommand{\TG}{\mathcal{T}}

\author{Lies Beers}
\author{Magnus Bakke Botnan}
\usepackage{mathtools}
\usepackage{tikz}

\affil{Vrije Universiteit Amsterdam, Amsterdam, The Netherlands}

\begin{document}

\maketitle

\begin{abstract}
We investigate several problems concerning extremal Betti numbers and persistence in filtrations of flag complexes. For graphs on $n$ vertices, we show that $\beta_k(X(G))$ is maximal when $G=\mathcal{T}_{n,k+1}$, the Turán graph on $k+1$ partition classes, where $X(G)$ denotes the flag complex of $G$. Building on this, we construct an edgewise (one edge at a time) filtration $\mathcal{G}=G_1\subseteq \cdots \subseteq \mathcal{T}_{n,k+1}$ for which $\beta_k(X(G_i))$ is maximal for all graphs on $n$ vertices and $i$ edges.  Moreover, the persistence barcode $\mathcal{B}_k(X(G))$ achieves a maximal number of intervals, and total persistence, among all edgewise filtrations with $|E(\mathcal{T}_{n,k+1})|$ edges.

For $k=1$, we consider edgewise filtrations of the complete graph $K_n$. We show that the maximal number of intervals in the persistence barcode is obtained precisely when $G_{\lceil n/2\rceil \cdot \lfloor n/2 \rfloor}=\mathcal{T}_{n,2}$. 
Among such filtrations, we characterize those achieving maximal total persistence. We further show that no filtration can optimize $\beta_1(X(G_i))$ for all $i$, and conjecture that our filtrations maximize the total persistence over all edgewise filtrations of $K_n$.
\end{abstract}

\section{Introduction}
\label{sec:intro}
A central theme in topological data analysis (TDA) is the computation of homological invariants from simplicial complexes. These complexes are often constructed from point cloud data, with the \emph{Vietoris--Rips} complex being one of the most widely used constructions. For a finite set \( P \) in a metric space \( (M, d) \), the Vietoris--Rips complex at scale \( r \), denoted \( \mathrm{VR}_r(P) \), includes all simplices \( \sigma = \{p_0, \ldots, p_n\} \) such that \( d(p_i, p_j) \leq r \) for all \( 0 \leq i, j \leq n \). This complex is a \emph{flag complex}, meaning it is the largest simplicial complex with a given underlying graph (the 1-skeleton). In particular, the edges in the connectivity graph of \( P \) at scale $r$ fully determine the higher-dimensional simplices in ${\rm VR}_r(P)$. Varying the scale parameter \( r \) induces a filtration of simplicial complexes:
\[
\mathrm{VR}(P)_{r_0} \hookrightarrow \mathrm{VR}(P)_{r_1} \hookrightarrow \cdots \hookrightarrow \mathrm{VR}(P)_{r_m},
\]
where \( r_0 < r_1 < \cdots < r_m \). Applying \( k \)-dimensional homology over a field \( \mathbf{k} \) to this filtration yields a persistence barcode in degree \( k \), denoted \( \mathcal{B}_k(\mathrm{VR}(P)) \). This barcode, consisting of intervals \([a, b)\), encodes the birth and death of topological features across scales and serves as a powerful tool for extracting topological information from \( P \). When \( P \) is a sufficiently dense sampling of an underlying space \( X \), the number of ``long bars'' in \( \mathcal{B}_k(\mathrm{VR}(P)) \) determines the \( k \)-th Betti number \( \beta_k(X) \), which describes the \( k \)-dimensional topological features of \( X \) \cite{chazal2008towards}. For a comprehensive introduction to topological data analysis, we refer the reader to \cite{dey2022computational}.

Given the central role of such filtrations in TDA, this paper addresses a fundamental question: how many topological features can arise in a data set of \( n \) points? Specifically, we investigate bounds on quantities such as the maximal value of \( \beta_k(\mathrm{VR}_r(P)) \), the maximal number of intervals in \( \mathcal{B}_k(\mathrm{VR}(P)) \), the length of the longest interval, and the sum of the lengths of the intervals (the total persistence).

\begin{remark}
In this paper, we focus on edgewise filtrations of graphs, i.e., sequences of graphs $ \mathcal G = 
G_1 \subseteq G_2 \subseteq \cdots \subseteq G_m = \bigl\{G_i\bigr\}_{i=1}^m,$
on $n$ vertices, where $G_{i+1}$ and $G_i$ differ by precisely one edge. Notably, every filtration of flag complexes can be realized as the Vietoris--Rips filtration of an appropriate metric on the vertices of $G_m$; see \cref{sec:appVR}. Thus, our results are directly applicable to data analysis using persistent homology.
\end{remark}

\subsection{Overview and contributions}
In \cref{sec:background}, we introduce the relevant background material and compute the Betti numbers of flag complexes on Turán graphs.

In \cref{sec:maxbetti}, we examine extremal Betti numbers and establish tight upper bounds on $\beta_k(F)$, where $F$ is a flag complex on $n$ vertices (\cref{thm:turanopt}). This upper bound is achieved by the Turán graph $\mathcal{T}_{n,k+1}$.

In \cref{sec:optimalfiltration}, we study filtrations of flag complexes on $n$ vertices with at most $e$ edges, where $e$ is the number of edges in $\mathcal{T}_{n,k+1}$. We prove that our filtration maximizes $\beta_k$ at all filtration steps (\cref{thm:turanoptfilt}).

In \cref{sec:barlength}, we analyze the longest possible interval in the barcode of a filtration of flag complexes on $n$ vertices (\cref{cor:barlengths}).

Additionally, in \cref{sec:maxintervals}, we focus on homology degree $1$ and show that any filtration of flag complexes on $n$ vertices will achieve the maximal number of intervals in its barcode in degree $1$ if and only if the filtration contains $\mathcal{T}_{n,2}$ (\cref{cor:maxnmbrbars}).

In \cref{sec:liesmagic}, we explore flag filtrations in homology degree $1$ that achieve both a maximal number of intervals and maximal total persistence. Equivalently, we maximize the total persistence of all filtrations of graphs containing $G_{\lceil n/2\rceil \cdot \lfloor n/2 \rfloor} = \mathcal{T}_{n,2}$. Given the importance of Turán graphs in extremal graph theory, we believe that identifying extremal filtrations of Turán graphs is interesting in its own right. Our result relies on an elaborate combinatorial analysis that precisely classifies the extremal filtrations (\cref{thm:lies}).

Finally, the paper concludes with a discussion in which we outline several important conjectures for future work.

\subsection{Related work}
\label{sec:relatedwork}
The study of extremal values of $\mathbb{Z}$-linear functions on the $f$- and $\beta$-vectors of simplicial complexes with $n$ vertices has a rich history. Classical questions include determining the extremal Euler characteristic and the maximal sum of Betti numbers. These problems were addressed for general simplicial complexes in \cite{bjorner1988extended} and later extended to arbitrary $\mathbb{Z}$-linear functions in \cite{kozlov1997convex}. Specifically, \cite[Theorem 3.4]{kozlov1997convex} shows that for flag complexes, extremal values are always realized by $\mathcal{T}_{n,k}$ for some $k$. Consequently, \cref{thm:turanopt} follows directly from \cite[Theorem 3.4]{kozlov1997convex}. 

Our proof of \cref{thm:turanopt} adapts the approach of \cite[Theorem 1.1]{adamaszek2014extremal}, which gives an alternative proof for the extremal total Betti number of flag complexes. The key observation in our proof of \cref{thm:turanopt} plays a central role in our results on extremal interval lengths (\cref{sec:barlength}). 

Another related work is \cite{goff2011extremal}, which examines asymptotic bounds for $\beta_k$ in Vietoris--Rips complexes of point samples in $\mathbb{R}^d$. Similar questions have been explored for Čech complexes; see, e.g., \cite{edelsbrunner2024maximum}. 

Our work diverges from earlier work by considering \emph{filtered} flag complexes. This approach is closely tied to the problem of finding extremal values for complexes with precisely $n$ vertices and $e$ edges; we return to this in \cref{sec:discussion}. Importantly, and in contrast to classical work, graphs achieving extremal values need not have $\mathcal{T}_{n,k}$ as a subgraph.

\section{Background}\label{sec:background}
{\bf Graphs and Simplicial Complexes:}
For a graph $G = (V,E)$, we write $\{v,w\}$ for the edge connecting the vertices $v$ and $w$. For a vertex $v\in V$, we let $N_G(v)\coloneqq\bigl\{w\colon \{v,w\}\in E\bigr\}$ denote the set of neighbors of $v$, and we let $N_G[v] \coloneqq N_G(v) \cup \{v\}$. The \emph{degree} of $v$ is $\deg(v)\coloneqq \deg_G(v) \coloneqq |N_G(v)|$, and $K_n$ denotes the complete graph on $n$ vertices. 

The \emph{complement graph} of $G$ is the graph $\overline G$ on the same vertices as $G$ where two distinct vertices of $\overline G$ are adjacent if and only if they are not adjacent in $G$. The \emph{join} of disjoint graphs $G_1$ and $G_2$ is the graph $G_1\vee G_2$ with vertex set $V(G_1)\sqcup V(G_2)$ and edge set $E(G_1)\cup E(G_2) \cup \bigl\{\{v_1,v_2\} \colon v_1\in V(G_1), v_2\in V(G_2)\bigr\}$.
The \emph{complete bipartite graph} $K_{n_1,n_2}$ is the graph $G_1\vee G_2$, where $G_i$ is the empty graph with $n_i$ vertices.
The sets $V(G_1)$ and $V(G_2)$ are called the \emph{partition classes} of $K_{n_1,n_2}$.

For a vertex $v$ in a simplicial complex $K$, the \emph{link} and \emph{(closed) star} of $v$ are the simplicial complexes given respectively by
\[\lk_K(v) =\{\tau \in K : v\not\in \tau, \{v\}\cup \tau\in K\}\qquad \qquad \stt_K(v) =\{\tau \in K : \{v\}\cup \tau\in K\}.\]
We write $K-v$ for the simplicial complex with simplices $\tau\in K$ for which $v\not\in K$. 

For a graph $G$, we let $X(G)$ denote the simplicial complex with $m$-simplices given by the $(m+1)$-cliques in $G$.  A simplicial complex $K$ is called a \emph{flag complex} if $K=X(G)$ for some graph $G$. The \emph{independence complex} of $G$ is the simplicial complex $\Ind(G) = X(\overline G)$.
Examples of flag and independence complexes are given in Figure \ref{fig:complexexample}. 
\begin{figure}
    \centering
    \includegraphics[scale=0.75]{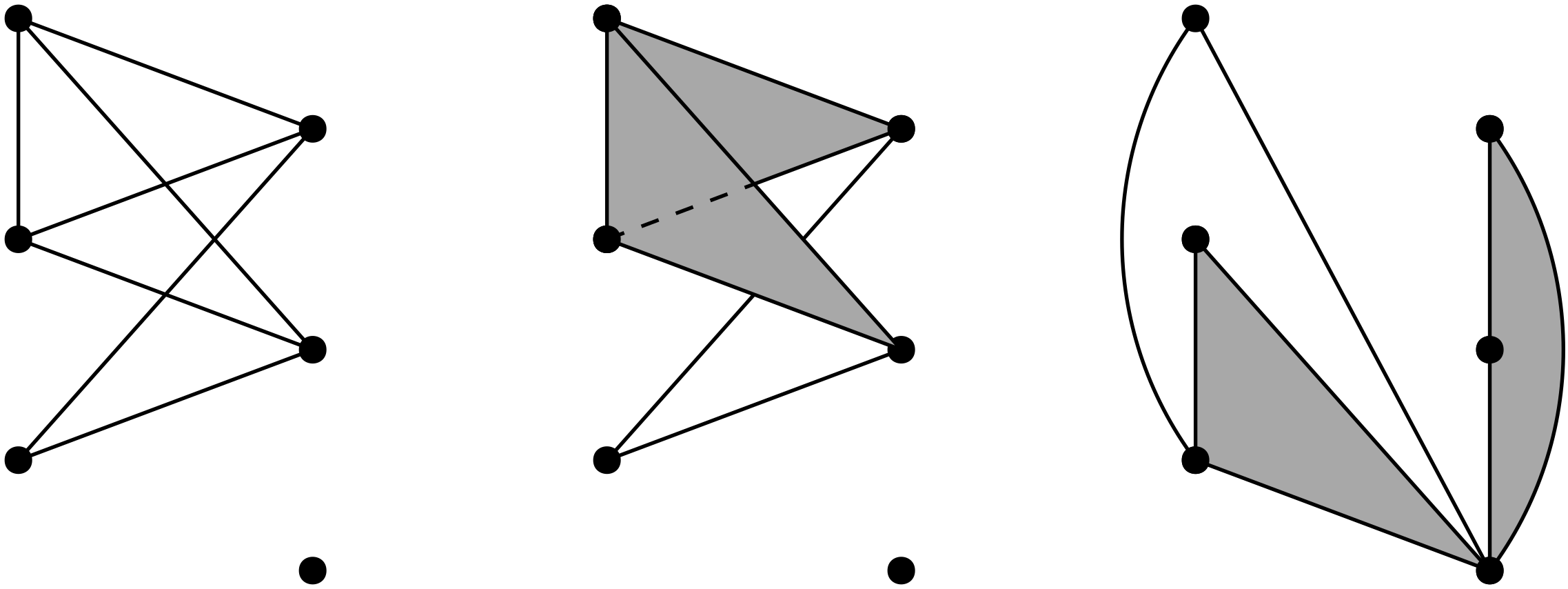}
    \caption{A graph (left), its flag complex (middle) and independence complex (right).}
    \label{fig:complexexample}
\end{figure}

{\bf Homology:}
For a simplicial complex $K$, we let $\beta_k(K)$ denote the dimension of the reduced homology group $\tilde{H}_k(K; \mathbf{k})$, and $C_k(K)$, denotes the vector space of $k$-chains. We use coefficients in a fixed, but arbitrary, field $\mathbf{k}$; our optimal constructions are torsion-free and thus our results do not depend on the choice of coefficient field. For that reason, we simply write $\tilde{H}_k(K)$. We also employ the notation $\BF_k(G) = \beta_k(X(G))$ for a graph $G$.

{\bf Persistent Homology:} A filtration $\mathscr{K}$ of simplicial complexes is a collection of simplicial complexes $\{K_i\}_{i=1}^m$ such that $K_i\subseteq K_j$ for $i\leq j$. Applying $\tilde{H}_k(-;\mathbf{k})$ to a filtration yields a sequence of vector spaces and linear maps $\tilde{H}_k(\mathscr{K}):\tilde{H}_k(K_1)\to \cdots \to \tilde{H}_k(K_m)$ called a \emph{persistence module}. Provided all the vector spaces are finite-dimensional, $\tilde{H}_k(\mathscr{K})$ is uniquely described by a collection of intervals in $\{1, \ldots, m\}$, called the \emph{degree $k$ barcode of $\mathscr{K}$}. We shall denote this barcode by $\mathcal{B}_k(\mathscr{K})$. The \emph{total persistence} of $\mathcal{K}$ is given by 
\[T\beta_k(\mathcal{K})= \sum_{[a,b)\in \mathcal{B}_k(\mathcal{K})} \bigl(b-a\bigr) = \sum_{i=1}^m \beta_k(K_i). \]

For a more thorough introduction to persistent homology, (generalized) persistence modules, and examples, see e.g., \cite[Chapter 3]{dey2022computational} or \cite{botnan2023introduction}.

For a filtration $\mathcal{G}$ of graphs (1-dimensional simplicial complexes), we get a filtration $X(\mathcal{G})$ of simplicial complexes by taking the flag complex at every index. If $G_{i+1}-G_i$ is a single edge for all $i$, then we say that the filtration $\mathcal{G}$ is \emph{edgewise}. We shall employ the notation $\BC_k(G) = \mathcal{B}_k(X(\mathcal{G}))$. 
\begin{example}
    An edgewise filtration of $K_5$ can be found in Figure \ref{fig:filtK5}.
    \begin{figure}
        \centering
        \includegraphics[scale=0.75]{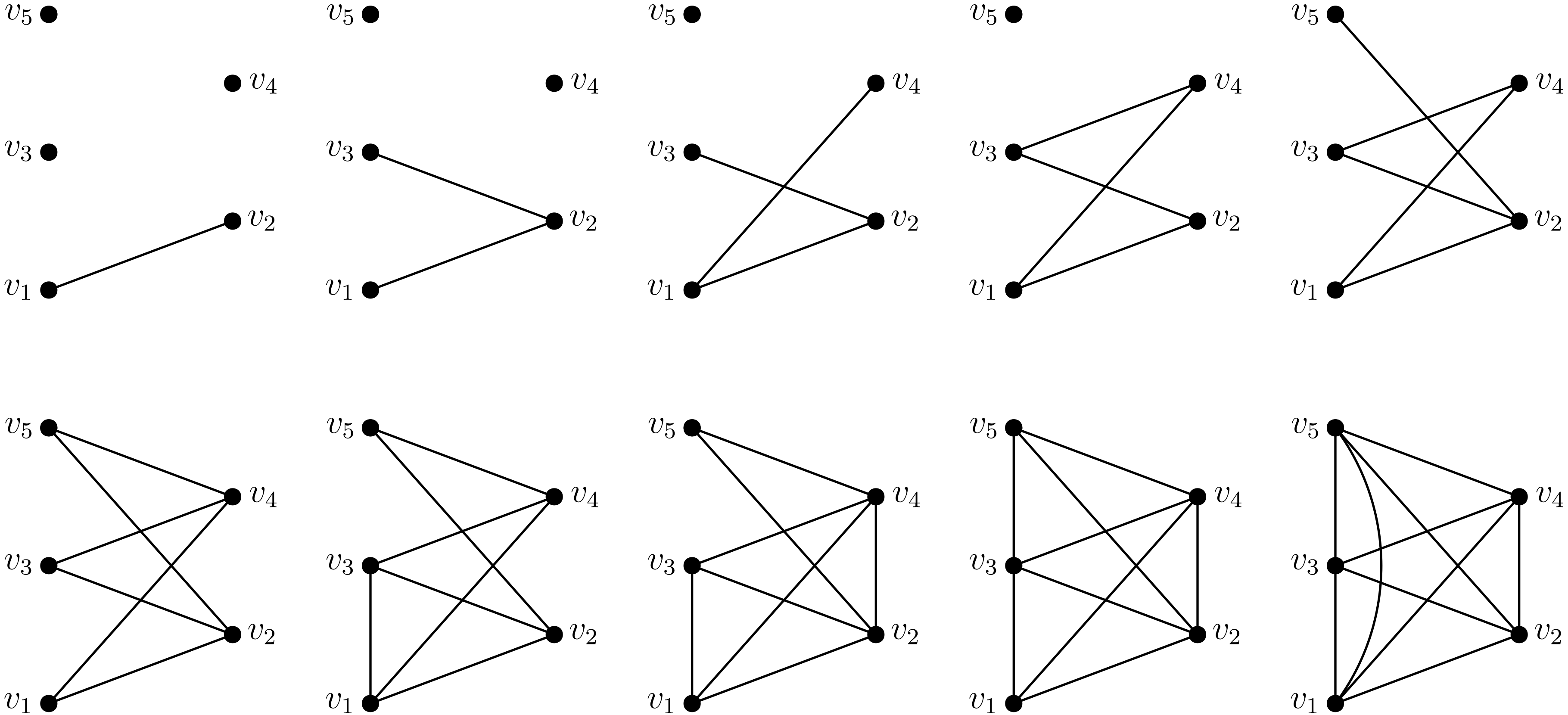}
        \caption{An edgewise filtration of $K_5$.}
        \label{fig:filtK5}
    \end{figure}
\end{example}

\subsection{Elementary homological properties}
The following two lemmas are well-known and important in combinatorial topology. For completeness, their proofs can be found in \cref{subsec:appeelementary}.
\begin{lemma}
\label{lem:MV-lk}
Let $K$ be a simplicial complex and $v\in V(K)$ a vertex. Then, for all $k\geq 1$,\[ \beta_k(K) \leq \beta_k(K-v) + \beta_{k-1}(\lk_K(v)).\]
\end{lemma}

\begin{lemma}
\label{lem:disjointunion}
    For any two graphs $G$ and $H$, and $k\geq -1$, 
    \[ \beta_k(\Ind(G\sqcup H)) = \sum_{i,j\geq -1; i+j=k-1} \beta_i(\Ind(G))\beta_j(\Ind(H)).\]
\end{lemma}

The following observation is essential for our work in \cref{sec:liesmagic}.
\begin{proposition}\label{prop:compbip}
Let $V(G)=V_1 \sqcup V_2$ be the vertex set of a graph $G$ containing all edges of the form $\{v,w\}$ where $v\in V_1$ and $w\in V_2$. Let $d_i \geq 1$ denote the number of connected components of $G$ restricted to $V_i$. Then, $\BF_1(G) = \BF_0(G_1)\BF_0(G_2) = (d_1-1)(d_2-1).$
\end{proposition}
\begin{proof}
Let $\overline G$ denote the complement graph of $G$ and observe that $\overline G=\overline G_1\sqcup \overline G_2$ where $G_1$ and $G_2$ are the full subgraphs of $G$ with vertices $V_1$ and $V_2$, respectively. From \cref{lem:disjointunion}, 
\[\BF_1(G) = \beta_1(\Ind(\overline G)) = \beta_0(\Ind(\overline G_1))\beta_0(\Ind(\overline G_2)) = \beta_0(X(G_1))\beta_0(X(G_2)). \qedhere\]
\end{proof}

\subsection{Turán graphs}
\begin{definition}\label{def:turan}
    Let $n\geq0$ and $k\geq1$, and let $n'$ be the smallest positive integer such that $n'\equiv n \bmod k$. Then, the \emph{$(n,k)$  graph} $\mathcal{T}_{n,k}$ is the complement graph of the graph 
    \[\overline{\mathcal{T}_{n,k}} = \bigsqcup_{i=1}^k K_{n_i}, \qquad 
n_i=
\begin{cases}
        \lceil n/k\rceil& \text{ if } 1\leq i\leq n',\\
        \lfloor n/k\rfloor & \text{otherwise.}
    \end{cases}\]
\end{definition}
\begin{example}
    See Figure \ref{fig:turangraphs} for some examples of Turán graphs.
    \begin{figure}
        \centering
        \includegraphics[scale=0.75]{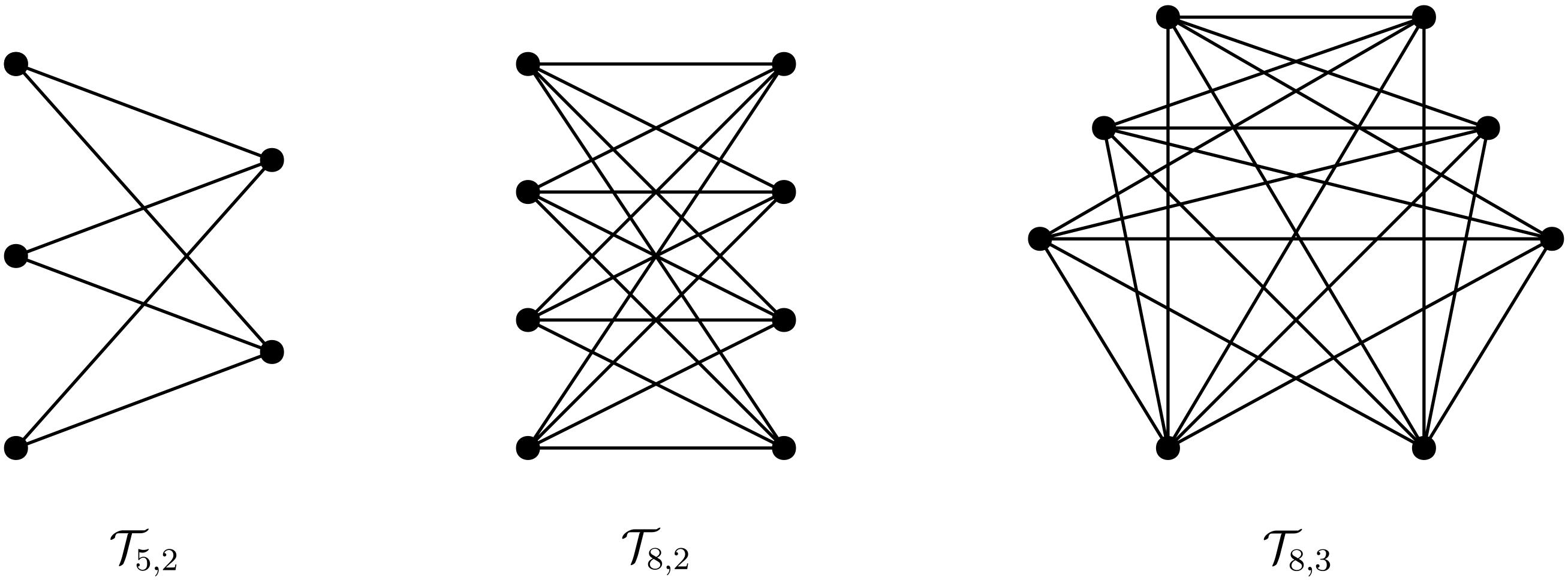}
        \caption{From left to right: the Turán graphs $\mathcal T_{5,2}$, $\mathcal T_{8,2}$ and $\mathcal T_{8,3}$.}
        \label{fig:turangraphs}
    \end{figure}
\end{example}
The proof of the following can be found in \cref{subsec:appeelementary}.
\begin{proposition}
    For all integers $n\geq 1$ and $k\geq 1$, we let $n'$ be the smallest positive integer such that $n'\equiv n \bmod k$. We have
    \[\BF_i(\TG_{n,k}) = \begin{cases}
        (\lceil n/k\rceil-1)^{(n')}\cdot (\lfloor n/k\rfloor-1)^{k - n'} & \text{if $i=k-1$},\\
        0 & \text{otherwise.}
    \end{cases}\]
In particular, if $n$ is a multiple of $k$, then 
$\beta_{k-1} = (n/k-1)^k.$
\label{prop:turanbetti}
\end{proposition}

\section{The maximum value of \texorpdfstring{$\beta_k(X(G))$}{bkX(G)}}
\label{sec:maxbetti}
The following is well-known; see \cref{sec:maxprodproof} for a proof.
\begin{lemma}
\label{lem:maxprod}
    Let $S$ be a positive integer, and $x_i\geq 1$ be integers satisfying $\sum_{i=1}^n x_i=S$. Then, 
 $\prod_{i=1}^n (x_i-1)\leq \prod_{i=1}^n (y_i-1)$ where
    \begin{equation}\label{eq:prod}
    y_i = 
    \begin{cases}
        \lceil S/n\rceil& \text{ if } 1\leq i\leq S\bmod n\\
        \lfloor S/n\rfloor & \text{otherwise.}
    \end{cases}
    \end{equation}
\end{lemma}
As said in \cref{sec:relatedwork}, the following proof is a modification of the proof of \cite[Theorem 1.1]{adamaszek2014extremal}.
\begin{theorem}
\label{thm:turanopt}
Let $G$ be a graph on $n$ vertices. Then, for all $k\geq 0$,
\[\BF_k(G) \leq (\lceil n/(k+1)\rceil-1)^{(n\bmod k+1)}\cdot (\lfloor n/(k+1)\rfloor-1)^{(k+1) - (n\bmod k+1)}.\]
\end{theorem}
\begin{proof}
We shall work inductively on $n$. The result is trivially true for $n=1$, so assume that it holds for all $n'<n$.

Note that $X(G) = \Ind(\overline G)$. Let $d$ denote the minimal degree over all $v\in V(\overline{G})$. Assume that $v$ is a vertex with $\deg(v) = d$, and let $\{v_1, \ldots, v_d\}$ denote the neighbours of $v$ in $\overline{G}$. Moreover, let $\overline{G}_i = \overline{G}-\{v_1, \ldots, v_i\}$ for $i=1, \ldots, d$ and $\overline{G}_0 = \overline{G}$. Applying \cref{lem:MV-lk},
\begin{align*}
\beta_k(\Ind(\overline{G})) &\leq \beta_k(\Ind(\overline{G}_1)) + \beta_{k-1}(\Ind(\overline{G}-N_{\overline{G}}[v_1])\\
& \leq \beta_k(\Ind(\overline{G}_2)) + \beta_{k-1}(\Ind(\overline{G}_1-N_{\overline{G}_1}[v_2])) + \beta_{k-1}(\Ind(\overline{G}-N_{\overline{G}}[v_1])))\\
& \vdots \\
& \leq \beta_k(\Ind(\overline{G}_d)) + \sum_{i=0}^{d-1} \beta_{k-1}(\Ind(\overline{G}_i-N_{\overline{G}_i}[v_{i+1}]))
\end{align*}
Here $\overline{G}_d$ has an isolated vertex and thus $\Ind(\overline{G}_d)$ is a cone, and therefore $\beta_k(\Ind(\overline{G}_d)) = 0$. Since every vertex of $\overline{G}$ has degree at least $d$, it follows that $\overline{G}_i - N_{\overline{G}_i}(v_{i+1})$ contains at most $n-(d+1) = n-d-1$ vertices. 
By the induction hypothesis it follows that
\[\beta_k(\Ind(\overline{G})) \leq d \cdot (x_1-1)\cdots (x_{k}-1)=  ((d+1)-1) \cdot (x_1-1)\cdots (x_{k}-1)\]
for integers $x_i\geq 1$ for which $(d+1)+\sum_{i=1}^{k}x_i =d+1+n-d-1=n$. Hence, by \cref{lem:maxprod}, 
\[\beta_k(\Ind(\overline{G})) \leq (\lceil n/(k+1)\rceil-1)^{(n\bmod k+1)}\cdot (\lfloor n/(k+1)\rfloor-1)^{(k+1) - (n\bmod k+1)}.\qedhere\]
\end{proof}
Let $\mathbf{G}(n)=\bigl\{G \colon \text{ graph on $n$ vertices}\bigr\}$.
Combining \cref{thm:turanopt} and \cref{prop:turanbetti}, we have
\begin{corollary}
\label{cor:turtight}
Of all graphs on $n$ vertices, $\mathcal{T}_{n,k+1}$ maximizes the $k$-th Betti number. I.e.,
\[\max_{G\in \mathbf{G}(n)} \BF_k(G) = \BF_k(\mathcal{T}_{n,k+1}).\]
In particular, if $n$ is a multiple of $k$, then 
$\max_{G\in \mathbf{G}(n)} \BF_k(G) = (n/k-1)^k.$
\end{corollary}

\section{The maximum value of \texorpdfstring{$\beta_k(X(G))$}{bkX(G)} for \texorpdfstring{$|E(G)| \leq |E(\mathcal{T}_{n,k+1})|.$}{E(G) < E(Tn,k+1).}}
\label{sec:optimalfiltration}
In this section, we fix $n,k\geq1$, and the notation
\begin{equation}e_{n,k+1}\coloneqq |E(\mathcal T_{n,k+1})| \text{ and }\Delta_{m-1,m}^{k+1}\coloneqq e_m^{k+1}-e_{m-1}^{k+1}.\end{equation}
The following result follows from combining \cref{lem:MV-lk} and \cref{cor:turtight}, and will be integral in proving the main result of this section.
\begin{corollary}
\label{cor:maxvertex}
Let $v$ be a vertex of degree $d\geq 1$ in a graph $G$. Then, 
\[\BF_k(G) \leq \BF_k(G-v)+ \BF_{k-1}(\mathcal{T}_{d,k}).\]
\end{corollary}
\begin{proof}
If $K=X(G)$, then $K-v = X(G-v)$, and \cref{lem:MV-lk} implies that 
\[\BF_k(G) \leq \BF_k(G-v)+ \beta_{k-1}(\lk_{X(G)}(v)).\]
Now observe that $\lk_{X(G)}(v) = X(N_G[v])$, and since $\lk_{X(G)}(v)$ has $d$ vertices, it follows from \cref{cor:turtight} that $\beta_{k-1}(\lk_K(v)) \leq \BF_{k-1}(\mathcal{T}_{d,k})$.\qedhere
\end{proof}
We now define an edgewise filtration $\mathcal H^{n,k+1}=\mathcal H=\bigl\{H_i\bigr\}_{i=1}^{e_{n,k+1}}=\bigl\{H_i^{n,k+1}\bigr\}_{i=1}^{e_{n,k+1}}$ of the Turán graph $\mathcal T_{n,k+1}$. In this section, we shall show, for $m=1,\ldots,|\mathcal T_{n,k+1}|$, that $\BF_k(H_m)$ is maximal over all graphs with $m$ edges, i.e., that $\mathcal H$ is \emph{fiberwise optimal}.
\begin{figure}[ht!]
        \centering
        \includegraphics[scale=0.75]{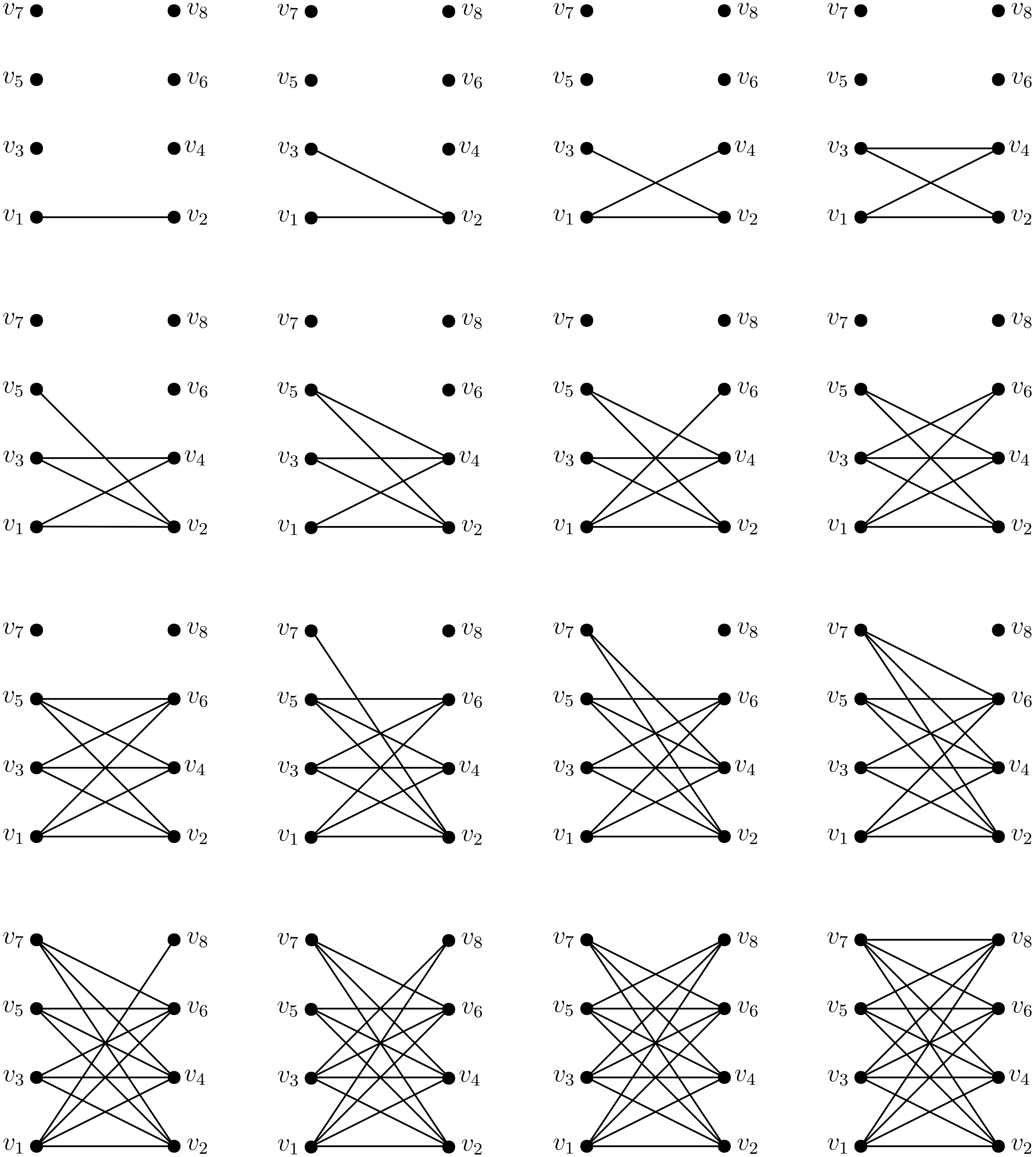}
        \caption{The filtration $\mathcal H$ of $\mathcal T_{8,2}$ from Definition \ref{def:H}.}
        \label{fig:filtT82}
    \end{figure}
\begin{definition}\label{def:H}
Let $V$ denote the vertices of $\mathcal{T}_{n,k+1}$ and
label the elements of $V$ such that $v_i$ and $v_{i+k+1}$ are in the same partition. Writing each  edge as $e_{i,j} = \{v_i, v_j\}$ for $i<j$, we order the edges by the following co-lexicographic order: $e_{i,j} < e_{k,l}$ if $\max\{i,j\}<\max\{k,l\}$, or if $i=k$ and $j<l$, or if $j=l$ and $i<k$.
Following this order, let 
$H_i=H_i^{n,k+1}$ denote the subgraph of $\mathcal{T}_{n,k+1}$ with the first $i$ edges.
\end{definition}
Note that, since $\mathcal T_{n,k+1}$ contains no $k+1$-simplices, $\beta_k(H_i)$ is increasing as a function of $i$.
\begin{example}
    See Figure \ref{fig:filtK5} for an edgewise filtration $\mathcal G = \bigl\{G_i\bigr\}_{i=1}^6$ of $K_5$, where $G_i\cong H_i^{5,2}$ for $i=1,\ldots,6$. Furthermore, in Figure \ref{fig:filtT82} one can find the filtration $\mathcal H^{8,2}$ of $\mathcal T_{8,2}$.
\end{example}

The next lemma shows that once a vertex has been connected to the growing component, and until the next vertex gets connected, the change in $\BF_k$ from adding a single edge, increases with the number of edges already added. The proof of the following Lemma is similar to that of \cref{cor:maxvertex} and can be found in \cref{sec:proofHe}.
\begin{lemma}\label{lem:He}
For $e$ edges, let $m$ be maximal such that $\mathcal{T}_{m,k+1}$ is a subgraph of $H_e$. Then, \[\BF_k\bigl(H_e\bigr) = \BF_k\bigl(\mathcal{T}_{m,k+1}\bigr) + \BF_{k-1}\bigl(\mathcal{T}_{e-e_m^{k+1}, k}\bigr).\]
In particular, for $e_m^{k+1}\leq e_1 < e_2 < e_{m+1}^{k+1}$,
we have that \[\BF_k(H_{e_1+1}) - \BF_k(H_{e_1})\leq \BF_k(H_{e_2+1}) - \BF_k(H_{e_2}).\]
\end{lemma}

\begin{theorem}\label{thm:turanoptfilt}
    Let $e\leq e_{n,k+1}$, let $H_e$ be as above, and let $G$ be any other graph on $n$ vertices and $e$ edges. Then, $\BF_{k}(G)\leq \BF_k(H_e).$
\end{theorem}

\begin{proof}
    We prove this inductively on the number of edges $e$. The statement is clearly true for $e=1$, so let's assume that the result holds for all $e'<e$. 
    
    Let $m$ be the number of vertices in $H_e$ with positive degree, and let $n'$ denote the number of vertices in $G$ with positive degree. If $n'<m$, then by \cref{thm:turanopt},
    \[\BF_k(G) \leq \BF_k(\mathcal{T}_{m-1,k+1}) \leq \BF_k(H_e).\]
    Hence, we may assume that $n'\geq m$. In particular, the average degree of the positive-degree-vertices in $G$ is no larger than the average degree of positive-degree-vertices in $H_e$.

    Choose a vertex $v$ in $G$ with minimal positive degree $d$, and observe that 
    \[d \leq e^{k+1}_{m} - e^{k+1}_{m-1} =\Delta_{m-1,m},\]
    since $\mathcal{T}_{m-1,k+1}\subsetneq H_e \subseteq \mathcal{T}_{m,k+1}$. In fact, if $d=\Delta_{m-1,m}$, then we must have that the average degree in $H_e$ is at least $d$. Importantly, this happens if and only if $H_e = \mathcal{T}_{m,k+1}$ and $m$ is a multiple of $k+1$. 

    By the induction assumption,
    \begin{align*}
        \BF_k(G-v) &\leq \BF_k(H_{e-\deg(v)}), \text{ and thus, by Corollary \ref{cor:maxvertex},}\\
        \BF_k(G) &\leq \BF_k(H_{e-\deg(v)})+ \BF_{k-1}(\mathcal{T}_{d,k}).
    \end{align*}
    If $d=\Delta_{m-1,m}$, then this becomes, by Lemma \ref{lem:He},
    \[\BF_k(G) \leq \BF_k(\mathcal{T}_{m-1,k+1}))+ \BF_{k-1}(\mathcal{T}_{d,k}) = \BF_k(\mathcal{T}_{m,k+1}) = \BF_k(H_e).\]
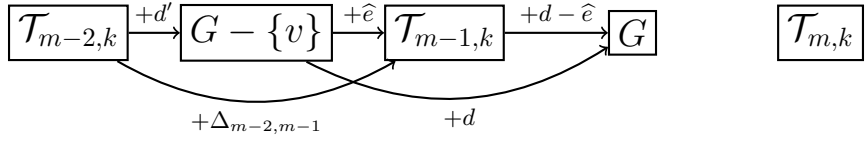
\begin{figure}
\centering
\begin{tikzpicture}[auto,node distance=2.5cm,
  thick,main node/.style={draw,font=\sffamily\Large\bfseries}]
  \node[main node] (1) {$\mathcal{T}_{m-2,k}$};
  \node[main node] (2) [right of=1] {$G-\{v\}$};
  \node[main node] (3) [right of=2] {$\mathcal{T}_{m-1,k}$};
    \node[main node] (4) [right of=3] {$G$};
    \node[main node] (5) [right of=4] {$\mathcal{T}_{m,k}$};

  \path[every node/.style={font=\sffamily\small}]

    (1) edge[->,bend right] node [below] {$+\Delta_{m-2,m-1}$} (3)
    (2) edge[->,bend right] node [below] {$+d$} (4)
    (1) edge[->] node [above] {$+d'$} (2)
    (2) edge[->] node [above] {$+\widehat e$} (3)
    (3) edge[->] node [above] {$+d-\widehat e$} (4);
\end{tikzpicture}
\caption{An illustration of Case $2$ from the proof of Theorem \ref{thm:turanoptfilt}.}
\label{fig:case2}
\end{figure}
We may therefore assume that 
$d\leq \Delta_{m-1,m}-1 \leq \Delta_{m-2,m-1}$. 

Let $\hat{e}$ denote the number of edges that is added to $H_{e-d}$ before a new vertex gets positive degree in the filtration $\mathcal H$. Let $d'$ be the degree of the last vertex in $H_{e-d}$ that obtained positive degree in $\mathcal H$. We consider two cases.
\begin{itemize}
\item {\bf Case 1: $\mathcal{T}_{m-1,k} \subseteq H_{e-d}$.} Then, 
\begin{align*}
\BF_k(G) &\leq \BF_k({H_{e-d}}) + \BF_{k-1}(\mathcal{T}_{d,k}) \\
&= \BF_k({H_{e-d}}) + \BF_{k}(H_{e_m^{k+1}+d})- \BF_{k}(H_{e_m^{k+1}})\\
&\leq \BF_k({H_{e-d}}) + \BF_{k}(H_{e})- \BF_{k}(H_{e-d})\\
&= \BF_k(H_e).
\end{align*}
\item {\bf Case 2: $\mathcal{T}_{m-2,k} \subseteq H_{e-d}\subset \mathcal{T}_{m-1,k}$.}
Note that $d'+\hat{e} = \Delta_{m-2,m-1} \geq d$ and $d-\hat{e} \geq 0$ (see Figure \ref{fig:case2}). In particular, 
\begin{align*}
\BF_k(G) &\leq \BF_k({H_{e-d}}) + \BF_{k-1}(\mathcal{T}_{d,k}) \\
&= \BF_k({H_{e-d}}) + \bigl(\BF_{k-1}(\mathcal{T}_{d,k})-\BF_{k-1}(\mathcal{T}_{d-\hat{e},k})\bigr)+ \BF_{k-1}(\mathcal{T}_{d-\hat{e},k})\\
&\leq  \BF_k({H_{e-d}}) + \bigl(\BF_{k-1}(\mathcal{T}_{d'+\hat{e},k}\bigr)- \BF_{k-1}(\mathcal{T}_{d',k})) + \BF_{k-1}(\mathcal{T}_{d-\hat{e},k})\\
&= \BF_k(H_e).
\end{align*}
\end{itemize}
In both cases, the inequality follows from \cref{cor:maxvertex} and \cref{lem:He}.
\end{proof}

\section{Tight bounds on the vanishing of homology and extremal interval lengths}
\label{sec:barlength}
For a graph $G$ with $n$ vertices, it is guaranteed that there exists a vertex of degree $(n-1)$ when the average degree exceeds $n-2$. Specifically, if the number of edges satisfies $m > \frac{n(n-2)}{2}$, then $X(G)$ must be a cone, implying that $\BF_k(G) = 0$ for all $k$. However, in practical scenarios, the primary interest is with $\BF_k(G)$ for small $k$. In this section, we provide tight bounds for the vanishing of $\BF_k(G)$ for a fixed $k$.

\begin{lemma}\label{lem:removingedges}
    Let $G$ be a graph with minimum degree $u$. Let $v$ be a vertex of degree $u$ and let $N_G(v)=\bigl\{v_1,\ldots,v_u\bigr\}$. Let $d_i\coloneqq \deg (v_i)$.
    If $G$ has $n$ vertices and $m$ edges, then
    \[
    \bigl| V\bigl(G - N_G[v_i]\bigr)\bigr| = n-d_i-1 \text{ and } \bigl| E\bigl(G - N_G[v_i]\bigr)\bigr| \leq m - \biggl( d_i + \biggl\lceil \frac{(u-1)d_i}{2}\biggr\rceil\biggr).
    \]
\end{lemma}
\begin{proof}
    Write $ \widehat G \coloneqq G - N_G[v_i]$.
    The fact that $\bigl| V\bigl(\widehat G\bigr)\bigr| = n-d_i-1$ is trivial. For the inequality, note that, by removing $N_G[v_i]$ from $G$, we remove the $d_i$ edges containing $v_i$, and all edges containing the vertices of $N_G(v_i)$. Those vertices all have degree at least $u$, which means that they have at least $u-1$ neighbors apart from $v_i$. Because it might be the case that $N_G\bigl(N_G[v_i]\bigr) = N_G[v_i]$, it follows that
    \[
    \bigl| E\bigl(\widehat G\bigr)\bigr| \leq m - \biggl( d_i + \biggl\lceil \frac{(u-1)d_i}{2}\biggr\rceil\biggr).\qedhere
    \]
\end{proof}

\begin{theorem}\label{thm:vanishing}
    Let $G$ be a graph with $n$ vertices and $m>{n-1 \choose 2}+ k$ edges. Then $\BF_k(G) = 0$. 
\end{theorem}
\begin{proof}
Let us first consider the case $k=0$. This case is immediate from the fact that the  maximal number of edges in a graph on $n$ vertices with an isolated vertex is ${n-1 \choose 2}$.

Working inductively on $k$, assume that result holds for all $k'<k$, and that $m>{n-1 \choose 2}+k$. 
Let $\overline{G}$ be the complement graph of $G$, and note that the number of edges $\overline{m} = |E(\overline{G})|$ satisfies 
\[\overline{m} = {n \choose 2} - m <{n \choose 2}-{n-1 \choose 2}  -k = n-1-k.\] From the proof of \cref{thm:turanopt}, we have that 
\begin{align*}
\BF_k(G) = \beta_k(\Ind(\overline{G})) \leq \sum_{i=0}^{d-1} \beta_{k-1}(\Ind(\overline{G}_i-N_{\overline{G}_i}[v_{i+1}])),
\end{align*}
where $\overline{G}_i = \overline{G}-\{v_1, \ldots, v_i\}$, and $\{v_1, \ldots, v_u\}$ are the neighbors of a vertex $v\in V(\overline{G})$ with minimal degree $u$. We shall show that all the terms in the sum are zero.

If we let $u'\geq u-i$ denote the minimal degree of a vertex in $V(\overline{G}_i-N_{\overline{G}_i}[v_{i+1}])$, 
\begin{align*}&\bigl| E\bigl(\overline{G}_i-N_{\overline{G}_i}[v_{i+1}])\bigr| \\
&\overset{(i)}{\leq} |E(\overline{G}_i)| - \biggl( \deg_{\overline{G}_i}(v_{i+1}) + \biggl\lceil \frac{(u'-1)\deg_{\overline{G}_i}(v_{i+1})}{2}\biggr\rceil\biggr)\\
& \overset{(ii)}{\leq} |E(\overline{G}_i)| - \biggl( u-i + \biggl\lceil \frac{(u-i-1)(u-i)}{2}\biggr\rceil\biggr) = |E(\overline{G}_i)|-\sum_{j=1}^{u-i}j\\
& \overset{(iii)}{\leq} |E(\overline{G})|- \sum_{j=u-i+1}^uj - \sum_{j=1}^{u-i}j =|E(\overline{G})|-\sum_{j=1}^uj = \overline{m}-\left(\frac{u(u+1)}{2}\right).
\end{align*}
Here, $(i)$ follows from \cref{lem:removingedges}, $(ii)$ from $\deg_{\overline{G}_i}(v_i)\geq u'$, and $(iii)$ from the fact that $\overline{G}_i$ is obtained by removing $i$ vertices from $\overline{G}$ with degree at least $u$ in $\overline{G}$. 
Now write 
\[n'=|V(\Ind(\overline{G}_i-N_{\overline{G}_i}[v_{i+1}]))| \leq n-u-1,\] and observe that  \[
|E(\Ind(\overline{G}_i-N_{\overline{G}_i}[v_{i+1}]))| \geq  {n' \choose 2} - \overline{m}+ u(u+1)/2.\] It follows that,
\begin{align*}
&|E(\Ind(\overline{G}_i-N_{\overline{G}_i}[v_{i+1}]))| - {n'-1 \choose 2} \\
&\geq {n' \choose 2} - \overline{m}+ u(u+1)/2 - {n'-1 \choose 2} \\
&=n'-1-\overline{m}+u(u+1)/2> n'-1-(n-1-k)+u(u+1)/2\geq k + u(u+1)/2 \geq k.
\end{align*}
Hence, $\beta_{k-1}(\Ind(\overline{G}_i-N_{\overline{G}_i}[v_{i+1}]))=0$ by the induction hypothesis. 
\end{proof}
\begin{remark}
Observe that the difference between the bound for a cone and the bound given in \cref{thm:vanishing} is $n/2 - k - 1$. While this difference is linear in the number of vertices, it can result in a significant reduction in the number of higher-dimensional simplices; a reduction which has the potential to speed up current implementations of persistent homology for flag complexes, e.g., Ripser \cite{bauer2021ripser}.
\end{remark}

In the following example, we show that that bounds in the previous theorem are tight.
\begin{example}
\label{ex:maxlengthbar}
    Let $n=(k+1)p$, and let $K_{1,p-1}$ be a star graph on $p$ vertices and $p-1$ edges, i.e., a central vertex connected to all other vertices. Observe that $\beta_0(\Ind(K_{1,p-1}))=1$ as one vertex is completely disconnected in the complement graph. If, $H = \bigsqcup_{i=1}^{k+1} K_{1,p-1}$, then it follows from repeated application of \cref{lem:disjointunion} that $\beta_k(\Ind(H)) \geq \prod_{i=1}^{k+1} \beta_0(\Ind(K_{1,p-1}))=1$. The number of edges in $\Ind(H)$ is
    \[{n \choose 2} - (k+1)(p-1) ={n \choose 2} - n + (k+1) = {n \choose 2} - (n-1)+k = {n-1 \choose 2} +k. \]
\end{example}

\begin{corollary}
\label{cor:barlengths}
    Let $\mathcal{G}$ be an edgewise filtration on $n$ vertices. Then for any $[a,b)\in \BC_k(\mathcal{G})$, we have \[2k(k+1) \leq a < b \leq {n-1 \choose 2} +k.\]
\end{corollary}
\begin{proof}
    The bound on $a$ follows from \cref{thm:turanoptfilt}, and the bound on $b$ from \cref{thm:vanishing}.
\end{proof}
It is straightforward to define a filtration $\mathcal{F}$ of the complement graph of $H$ from \cref{ex:maxlengthbar} such that $\BC_k(\mathcal{\mathcal{F}})$ contains the interval $[a,b)$ from \cref{cor:barlengths}.

\section{The maximum value of \texorpdfstring{$|\BF_1(\mathcal{G})|$}{b1FL(G)}.}
\label{sec:maxintervals}
\begin{proposition}
\label{prop:maxint}
    Let $\mathcal{G} = \{G_i\}_{i=1}^m$ be an edgewise filtration. Then, there exists a triangle-free subgraph $H$ of $G_m$ such that $\BF_1(H) \geq |\BC_1(\mathcal{G})|.$
\end{proposition}
\begin{proof}
Choose a representative cycle $c_{[a,b]}$ for each non-empty interval $[a,b)\in \BC_1(\mathcal{G})$, e.g., by running the standard algorithm for persistent homology. For a cycle $c$, let $e(c)$ denote the set of edges on which the cycle has a non-zero coefficient, and let $H$ be the subgraph of $G_m$ with edges given by 
\[\bigcup_{[a,b)\in \BC_1(\mathcal{G})} e(c_{[a,b)]}).\]
Note that the cycles $\{c_{[a,b)} : [a,b)\in \BC_1(\mathcal{G})\}$ are linearly independent (as elements of $C_1(G)$) by virtue of being representatives for non-trivial intervals. 

If $X(H)$ contains a 2-simplex $\tau = \{v_1, v_2, v_3\}$, then any cycle $c_{[a,b)}$ supported on the edge $\{v_2, v_3\}$ is homologous in $\tilde{H}_1(X(H))$ to the cycle $c'_{[a,b)} = c_{[a,b)} - \partial_2(\tau)$. Hence, we can remove $\{v_2, v_3\}$ from $H$, without reducing the 1st Betti number. Doing this for all triangles, we end up with a triangle-free graph $\hat{H}$ containing at least $|\BC_1(\mathcal{G})|$ linearly independent $1$-cycles. In particular, $\BF_1(\hat{H})\geq |\BC_1(\mathcal{G})|$. 
\end{proof}
\begin{remark}
    This argument does not apply to homology in higher degrees. For instance, the removal of an edge can result in the removal of many $2$-simplices, some of which are faces of $3$-simplices, and others that are not. 
\end{remark}
The previous result actually gives a short proof of \cref{cor:turtight} for the case $k=1$. We shall include this proof here, as it ensures uniqueness of the extremal complex.
\begin{proposition}
    For any graph on $n$ vertices, we have $\BF_1(G)\leq  \BF_1(\mathcal{T}_{n,2})$ with equality if and only if $G=\mathcal{T}_{n,2}$.
\end{proposition}
\begin{proof}
    Define any edgewise filtration $\mathcal{G} = \{G_i\}_{i=1}^m$ for which $G_m = G$. Then, by \cref{prop:maxint}, there exists a triangle-free graph $H$ such that $\BF_1(H)\geq \BC_1(\mathcal{G}) \geq \BF_1(G)$. Hence, it suffices to maximize $\BF_1(G)$ over triangle-free graphs. By the Euler-Poincaré formula,
    \[\BF_1(G) = -\BF_0(G) - 1 - |V(G)| + |E(G)| .\]
    In conclusion, we are seeking a triangle-free graph with a maximal number of edges, but this is well-known to be uniquely the  graph $\mathcal{T}_{n,2}$ by Turán's theorem \cite[Theorem 11.17]{chartrand2013first}.
\end{proof}

Combined we arrive at the following.
\begin{corollary}\label{cor:maxnmbrbars}
    If $\mathcal{G}$ is a filtration of $K_n$, then $|\BC_1(\mathcal{G})| \leq \BF_1(\mathcal{T}_{n,2})$ with equality if and only if $G_{\lceil n/2\rceil \cdot \lfloor n/2 \rfloor}=\mathcal{T}_{n,2}$.
\end{corollary}

\section{Extremal filtrations for degree 1 persistent homology}
\label{sec:liesmagic}
In this section, we consider edgewise filtrations $\mathcal G = G_1\subseteq \cdots \subseteq G_{\binom{n}{2}}$ of the complete graph $K_n$ that maximize the total persistence and the number of bars for degree $1$ homology.
We shall answer the following question: \\
\indent \emph{For a fixed number of vertices $n$ and for $k=1$, which edgewise filtrations of the complete \indent graph $K_{n}$ with a maximal number of bars achieve maximal total persistence?}\\
By Corollary \ref{cor:maxnmbrbars}, this translates to finding
\[
\max\biggl\{\sum_{i=1}^{n(n-1)/2}\BF_1(G_i) \colon \mathcal{G} = \bigl\{ G_i \colon 1\leq i\leq \binom{n}{2}, G_{|E(\mathcal T_{n,2})|}=\mathcal T_{n,2} \text{ and } G_{n(n-1)/2}=K_n\bigr\}\biggr\}.
\]
First, we give some preliminaries and establish some notation.
\begin{itemize}
    \item All graphs $G=(V,E)$ have $n$ vertices. The join of two graphs $G_1$ and $G_2$, notation $G_1\vee G_2$, is as in Section \ref{sec:background}.
    \item For a filtration $\mathcal G = \bigl\{G_i\bigr\}_{i=1}^m$, we write $\BT_1\bigl(\mathcal G\bigr) = \sum_{i=1}^{m}\BF_1(G_i)$ for its total persistence of degree $1$.
    \item Denote $e_n\coloneqq |E(\mathcal T_{n,2})|$. If $|E(G)|\geq e_n$, then $G$ has $\mathcal T_{n,2}$ as a subgraph, with partition classes $V_1$ and $V_2$ of sizes $|V_1| = \lceil n/2 \rceil$ and $|V_2| = \lfloor n/2 \rfloor$.
    \item Let $d_i$ be the number of connected components of $G|_{V_i}$. Then, by Proposition \ref{prop:compbip},
    \begin{equation}\label{eq:productcomponents}\BF_1(G)=(d_1-1)(d_2-1).\end{equation}
\end{itemize}
We also define what it means for two graphs and two filtrations to be (fiberwise) isomorphic.
\begin{definition}
    Two graphs $G_1$ and $G_2$ with vertex sets $V(G_i)$ and edge sets $E(G_i)$ are \emph{isomorphic}, notation $G_1\cong G_2$, if there is a bijection $\varphi\colon V(G_1)\to V(G_2)$, such that $\{u_1,u_2\}\in E(G_1)$ if and only if $\{\varphi(u_1),\varphi(u_2)\}\in E(G_2).$
    Furthermore, two filtrations $\mathcal G^1 = \bigl\{G_i^1\bigr\}_{i=1}^m$ and $\mathcal G^2 = \bigl\{G_i^2\bigr\}_{i=1}^m$ are \emph{fiberwise isomorphic} if $G_i^1\cong G_i^2$ for all $i = 1,\ldots,m$.
\end{definition}
First, by Corollary \ref{cor:maxnmbrbars}, any filtration with a maximal number of bars must contain the Turán graph $\mathcal T_{n,2}$.
Moreover, Theorem \ref{thm:turanoptfilt} establishes that the filtration $\mathcal H = \bigl\{H_i\bigr\}_{i=1}^{e_n}$ from Definition \ref{def:H} is fiberwise optimal, reducing our problem to maximizing the total persistence of an edgewise filtration of the complete graph, beginning with the Turán graph plus one edge.
This is a bit more involved, because we cannot find a fiberwise optimal filtration of the complete graph, as is shown in the following example.
\begin{example}\label{ex:28}
    If $|E|=e_n+2$, then one gets optimal degree $1$ homology by adding one edge to $V_1$ and one edge to $V_2$ (left graph in Figure \ref{fig:ex28}). However, if $|E|=e_n+3$, then optimal degree $1$ homology is obtained by adding all three edges to $V_1$ to form a $K_3$ (second-to-left graph in Figure \ref{fig:ex28}). It is clear that a filtration of $K_n$ cannot contain both graphs.
\end{example}
\begin{figure}
        \centering
        \includegraphics[scale=0.75]{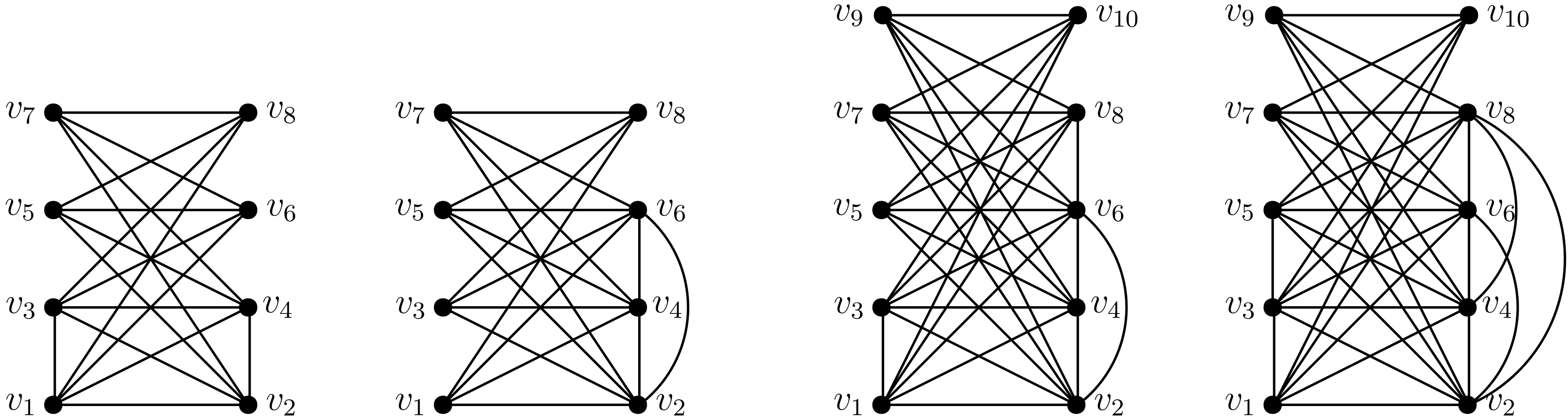}
        \caption{The graphs from Example \ref{ex:28} and \ref{ex:filt}.}
        \label{fig:ex28}.
    \end{figure}
We can prove the following about the structure of the graphs in an optimal filtration of $K_n$.
\begin{lemma}\label{lem:completegraphs}
    Let $\mathcal G$ be an edgewise filtration $\bigl\{G_i\bigr\}_{i=1}^{n\choose 2}$ of $K_n$ with maximal total persistence that includes $\mathcal T_{n,2}$. Then, for $t=1,2$ and $i=e_n+1,\ldots,\binom{n-1}{2}+1$, one of the subgraphs $G_i|_{V_t}$ consists of isolated vertices and a single component $K_m$, and the other subgraph $G_i|_{V_t}$ consists of isolated vertices, and a single component with $l$ vertices, such that $K_{l-1}$ is a subgraph of this component.
\end{lemma}
\begin{example}\label{ex:filt}
    Cf.\ Figure \ref{fig:ex28} for four examples of graphs with the properties from Lemma \ref{lem:completegraphs}.
\end{example}
\begin{proof}  
Suppose that for some $i$, $E(G_i) \setminus E(G_{i-1}) = \{w_1, w_2\}$, where $ w_1 $ and $ w_2 $ are isolated vertices in $ G_{j-1}|_{V_1} $ (without loss of generality). Also assume that $ G_{j-1}|_{V_1} $ contains a component $ C $ of size $ |C| > 1 $.
To maximize total persistence, it is more optimal to connect $w_1$ to $C$, since this allows us to add $|C|-1$  edges of the form $ \{w_1, c\} $ (where $ c \in C $) without increasing the number of connected components (cf.\ \eqref{eq:productcomponents}). Because these additional edges do not decrease the degree $1$ homology, we should add them as early in the filtration as possible to maximize total persistence. Hence, by this same argument, $ G_{j-1}|_{V_2}$ should be isomorphic to $K_m$ and some isolated vertices.

Furthermore, note that we added the restriction $j\leq \binom{n-1}{2}+1$ because Theorem \ref{thm:vanishing} ensures that beyond this point, $\BF_1(G_j)=0$. This means that 
\begin{equation}G_{\binom{n-1}{2}+1}=(K_1\sqcup K_{\lceil n/2 \rceil-1})\vee (K_1\sqcup K_{\lfloor n/2\rfloor -1}),\label{eq:grensgraaf}\end{equation}
but if $j>\binom{n-1}{2}+1$, then $G_j$ could be any graph containing $G_{\binom{n-1}{2}+1}$ as a subgraph.
\end{proof}
Because of this, we restrict our focus to edgewise filtrations of the graph $G_{\binom{n-1}{2}+1}$ from \eqref{eq:grensgraaf}.

Moreover, by Lemma \ref{lem:completegraphs}, we only have to consider filtrations of the described form. We can represent them by a sequence of tuples of the form
\begin{equation}\label{eq:repcompgrphs}
    \bigl(K_{1},K_{1}) = \bigl(K_{l_1},K_{r_1}),\bigl(K_{l_2},K_{r_2}),\ldots,\bigl(K_{l_c},K_{r_c}) = \bigl(K_{\lceil n/2\rceil-1},K_{\lfloor n/2 \rfloor-1}\bigr),
\end{equation}
such that $l_i\geq l_{i-1}$ and $r_i\geq r_{i-1}$.
The representation from \eqref{eq:repcompgrphs} corresponds to the filtration $\mathcal G = \bigl\{G_i\bigr\}_{i=1}^m$, where, for $i\geq1$,
\begin{align*}
G_{e_n+|E(K_{l_i})|+|E(K_{r_i})|} &\cong \biggl( K_{l_i}\sqcup \bigsqcup_{j=l_i+1}^{\lceil n/2 \rceil}K_1 \biggr) \bigvee \biggl(K_{r_i}\sqcup \bigsqcup_{j=r_i+1}^{\lfloor n/2 \rfloor}K_1\biggr) \text{ and }\\
G_{e_n+|E(K_{l_{i+1}})|+|E(K_{r_i})|} &\cong \biggl( K_{l_{i+1}}\sqcup \bigsqcup_{j=l_i+1}^{\lceil n/2 \rceil}K_1 \biggr) \bigvee \biggl(K_{r_i}\sqcup \bigsqcup_{j=r_i+1}^{\lfloor n/2 \rfloor}K_1\biggr).
\end{align*}
In other words, the representation from \eqref{eq:repcompgrphs} means that, starting from the Turán graph $\mathcal T_{n,2}$:
\begin{enumerate}
    \item we begin adding edges by forming a $K_{l_2}$ in $V_1$,
    \item then add edges to form a $K_{r_2}$ in $V_2$,
    \item increase the sizes of the complete subgraphs in $V_1$ and $V_2$ by constructing a $K_{l_3}$ (containing the smaller $K_{l_2}$) in $V_1$, a $K_{r_3}$ in $V_2$, and so on.
\end{enumerate}
Note that this representation is not unique ($(K_3,K_1)$ is equivalent to $(K_2,K_1),(K_3,K_1)$).
\begin{figure}
        \centering
        \includegraphics[scale=0.75]{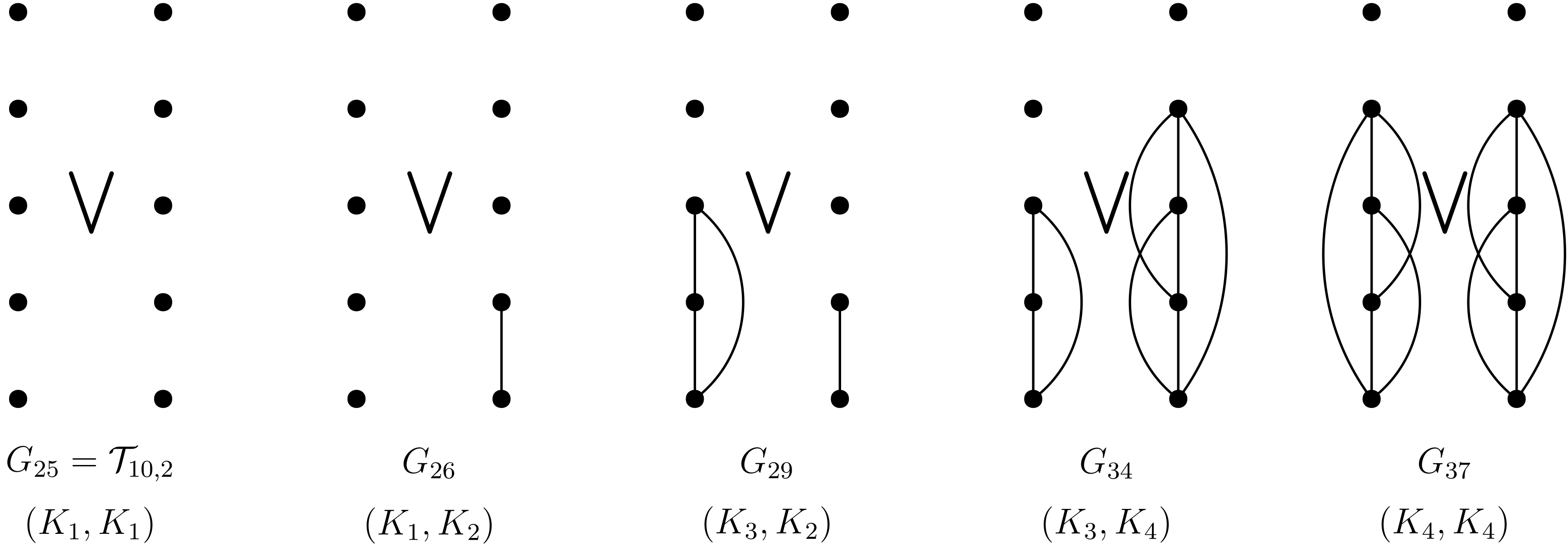}
        \caption{A filtration and its representation. Here, $\bigvee$ denotes the join of the left and right graphs.}
        \label{fig:exrep}
    \end{figure}
\begin{example}
    In Figure \ref{fig:exrep}, one finds a filtration and a corresponding representation.
\end{example}
\begin{figure}
        \centering
        \includegraphics[scale=0.75]{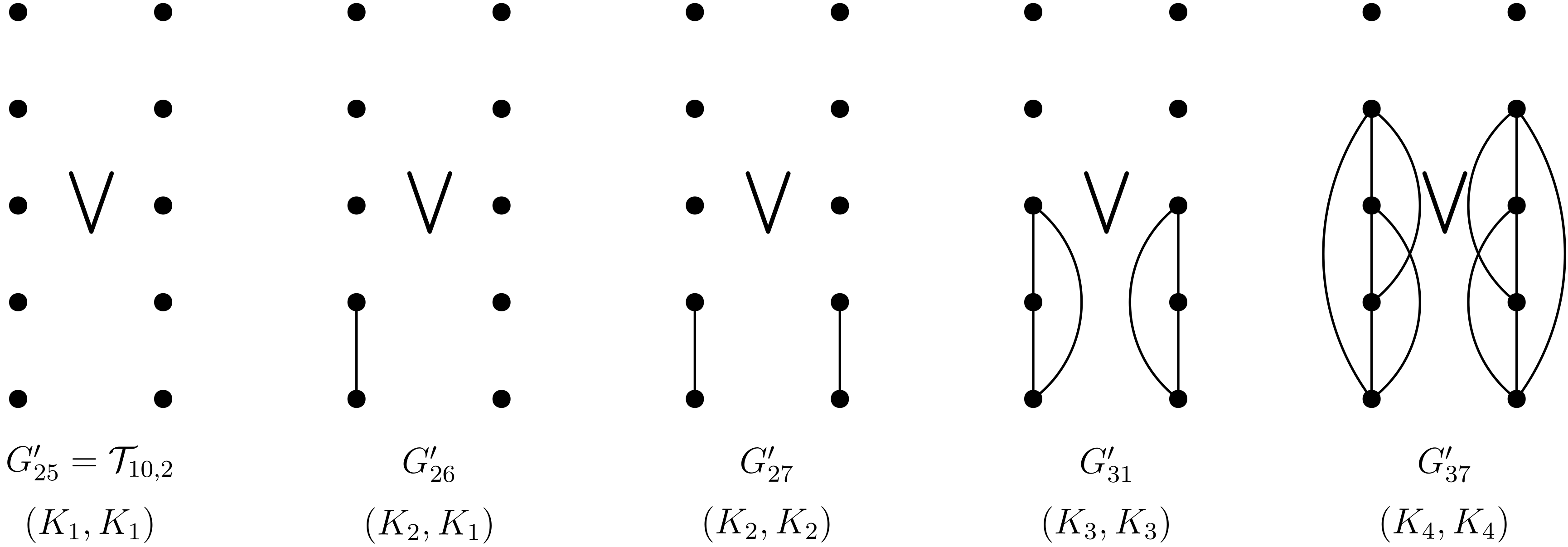}
        \caption{A filtration with the same total persistence as the one from Figure \ref{fig:exrep}.}
        \label{fig:lb}
    \end{figure}
Surprisingly, the optimal strategy is to first construct two large cliques in $V_1$ and $V_2$ of sizes approximately $\frac34|V_1|$ and $\frac34|V_2|$, respectively, and then, alternating between $V_1$ and $V_2$, increase the sizes of the cliques one by one. We prove this in the main result of this section.

\begin{theorem}
\label{thm:lies}
    Let $n\geq4$ and let $j_n\coloneqq \lfloor (3n-2)/8\rfloor$ and $k_n\coloneqq \lfloor (3n+7)/8\rfloor$. Up to fiberwise isomorphism, the edgewise filtration $\mathcal G_{n,\max}=\bigl\{G_i\bigr\}_{i=1}^{\binom{n-1}{2}+1}$ of $G_{\binom{n-1}{2}+1}$ with maximal total persistence and a maximum number of bars is given by $G_i\cong H_i$ (Definition \ref{def:H}) for $i=1,\ldots,e_n$. After $G_{e_n} = \mathcal T_{n,2}$, for $n\not\equiv 0 \mod 8$, the filtration is unique up to fiberwise isomorphism and represented by a sequence of tuples (cf.\ \eqref{eq:repcompgrphs}) as follows
    \[
    \begin{cases}
        \bigl(K_1,K_1\bigr),\bigl(K_{j_n},K_{j_n}\bigr), \bigl(K_{j_n+1},K_{j_n+1}\bigr),\ldots,\bigl(K_{n/2-1},K_{n/2-1}\bigr), &\text{ if }n\equiv 2,4,6\mod 8,\\
        \bigl(K_1,K_1\bigr),\bigl(K_{k_n},K_{k_n-1}\bigr),\bigl(K_{k_n+1},K_{k_n}\bigr),\ldots,\bigl(K_{(n-1)/2},K_{(n-3)/2}\bigr), &\text{ if }n\equiv 1 \mod 2.
    \end{cases}
    \]
    If $n\equiv 0\mod 8$, then there are, up to fiberwise isomorphism, two optimal filtrations, represented by
    \begin{align*}
        &\bigl(K_1,K_1\bigr),\bigl(K_{j_n},K_{j_n}\bigr), \bigl(K_{j_n+1},K_{j_n+1}\bigr),\ldots,\bigl(K_{n/2-1},K_{n/2-1}\bigr) \text{ and }\\
        &\bigl(K_1,K_1\bigr),\bigl(K_{j_n+1},K_{j_n+1}\bigr), \bigl(K_{j_n+2},K_{j_n+2}\bigr), \ldots,\bigl(K_{n/2-1},K_{n/2-1}\bigr).
    \end{align*}
\end{theorem}
\begin{remark}\label{rmk:leftahead}
We prove the even case. The proof of the odd case is more or less the same (although the values are slightly different). However, we need to prove that, for any representation as in \eqref{eq:repcompgrphs}, $l_i\geq r_i$ for $i=1,\ldots,c$.
In the even case, we may assume this.

See Figure \ref{fig:lb} for an example of and Appendix \ref{app:lies} for a proof of this assumption.
\end{remark}
\begin{proof}
    It follows from Corollary \ref{cor:maxnmbrbars} that, for a filtration $\mathcal G$ to have a maximum number of bars, we must have $G_{e_n}=\mathcal T_{n,2}$.
    For a filtration $\mathcal G$ to also have maximal total persistence, it immediately follows by Theorem \ref{thm:turanoptfilt} that $G_i\cong H_i$ for $i=1,\ldots,e_n$.
    
    Any filtration $\mathcal G$ that we consider in this proof will therefore be such that $G_i\cong H_i$ for $i=1,\ldots,e_n$. We are interested in the remainder of the filtration, which we can represent by a sequence of tuples as in \eqref{eq:repcompgrphs}.
    Within the proof, we will further reduce it by writing 
    \begin{equation}\label{eq:reducednot}
    \bigl(l_1,r_1\bigr),\bigl(l_2,r_2\bigr),\ldots, \bigl(l_c,r_c\bigr) \text{ instead of }\bigl(K_{l_1},K_{r_1}),\bigl(K_{l_2},K_{r_2}),\ldots,\bigl(K_{l_c},K_{r_c}).
    \end{equation}
    Furthermore, given a filtration $\mathcal G$ represented by $\bigl(l_1,r_1\bigr),\bigl(l_2,r_2\bigr),\ldots, \bigl(l_c,r_c\bigr)$, we define its \emph{alternation depth} $d_{\mathcal G}$ as follows:
    \begin{itemize}
        \item We let $d_{\mathcal G}\coloneqq 1$ if $\bigl(l_i,r_i\bigr)=\bigl(l_{i-1}+1,r_{i-1}+1\bigr)$ for all $i=2,\ldots,c$.
        \item Otherwise, let $d_{\mathcal G}\geq 2$ be such that $\bigl(l_i,r_i\bigr)=\bigl(l_{i-1}+1,r_{i-1}+1\bigr)$ for all $i=d_{\mathcal G}+1,\ldots,c$, but $\bigl(l_{d_{\mathcal G}},r_{d_{\mathcal G}}\bigr)\neq \bigl(l_{d_{\mathcal G}-1}+1,r_{d_{\mathcal G}-1}+1\bigr)$.
    \end{itemize}
    The idea of the proof is to start with a filtration $\mathcal G = \bigl\{G_i\bigr\}_{i=1}^{\binom{n-1}2+1}$, with a representation as in \eqref{eq:reducednot}, and show that we can change the filtration in steps, to obtain in every step a filtration $\mathcal G'$ such that $\BT_1\bigl(\mathcal G'\bigr)\geq \BT_1\bigl(\mathcal G\bigr)$, and end up with an optimal solution.
    
    Now, writing $n=2p$, let $\mathcal G$ be a filtration, represented by \[\bigl(1,1) = \bigl(l_1,r_1),\bigl(l_2,r_2\bigr),\ldots,\bigl(l_c,r_c) = \bigl(n/2-1,n/2-1\bigr)=\bigl(p-1,p-1\bigr).\]
    We shall show that
    \begin{enumerate}
        \item If the alternation depth $d_{\mathcal G}>j_n$,
        then we can change $\mathcal G$ to a filtration $\mathcal{G'}$ such that $d_{\mathcal G'}< d_{\mathcal G}$ and $\BT_1\bigl(\mathcal G'\bigr) \geq \BT_1\bigl(\mathcal G\bigr)$.
        \item If $l_2< j_n$, then we can change $\mathcal G$ to a filtration $\mathcal G'$ represented by $\bigl(l_1',r_1'\bigr),\ldots,\bigl(l_{c'}',r_{c'}'\bigr)$ with $l_2'=l_2+1$, such that $c'=c-1$ or $c'=c$ (depending on $\mathcal G$) and $\BT_1\bigl(\mathcal G'\bigr) \geq \BT_1\bigl(\mathcal G\bigr)$.
    \end{enumerate}
    For the first part, let $\mathcal G$ be a filtration with alternation depth $d_{\mathcal G}>j_n$. Its representation, substituting $n$ by $2p$, is of the form
    \[
    (1,1),\ldots,(t,u),(v,w),\bigl(d_{\mathcal G},d_{\mathcal G}\bigr),\bigl(d_{\mathcal G}+1,d_{\mathcal G}+1\bigr),\ldots,\bigl(p-1,p-1\bigr) \text{ with } \bigl(v,w\bigr)\neq \bigl(d_{\mathcal G}-1,d_{\mathcal G}-1\bigr).
    \]
    Since we assumed that $v\geq w$, we must have $w<d_{\mathcal G}-1$. We transform $\mathcal G$ into a filtration $\mathcal G'$ which is the same as $\mathcal G$, except that we add an extra alternation step. We consider two cases.
    \begin{itemize}
        \item Case $1$: $v<d_{\mathcal G}-1$. In this case, $\mathcal G'$ is represented by
        \[
        (1,1),\ldots (t,u),(v,w),\bigl(d_{\mathcal G}-1,d_{\mathcal G}-1\bigr), \bigl(d_{\mathcal G},d_{\mathcal G}\bigr),\bigl(d_{\mathcal G}+1,d_{\mathcal G}+1\bigr),\ldots,\bigl(p-1,p-1\bigr).
        \]
        \item Case $2$: $v=d_{\mathcal G}-1$. In this case, $\mathcal G'$ is represented by
        \[
        (1,1),\ldots (t,u),\bigl(v=d_{\mathcal G}-1,d_{\mathcal G}-1\bigr), \bigl(d_{\mathcal G},d_{\mathcal G}\bigr),\bigl(d_{\mathcal G}+1,d_{\mathcal G}+1\bigr),\ldots,\bigl(p-1,p-1\bigr).
        \]
    \end{itemize}
    In both cases, we see that $d_{\mathcal G'}<d_{\mathcal G}$. Moreover, as we show in more detail in Appendix \ref{app:lies},
    \begin{equation}\label{eq:apptoprove1}
        \BT_1\bigl(\mathcal G'\bigr) - \BT_1\bigl(\mathcal G\bigr) = \frac16\bigl(d_{\mathcal G}-1-w\bigr)\bigl(d_{\mathcal G}-w\bigr)\bigl(4d_{\mathcal G}+2w-3p-2\bigr)\geq0.
    \end{equation}
    To see that this difference is non-negative, first recall that $w<d_{\mathcal G}-1$, so the first two terms are strictly positive. Furthermore, since $d_{\mathcal G}>\lfloor 3p/4-1/4\rfloor$, we have $d_{\mathcal G}\geq \lceil 3p/4\rceil\geq 3p/4$, so
    \[
    4d_{\mathcal G}+2w-3p-2\geq 2w-2\geq0,
    \]
    since $w\geq 1$. Note that $\BT_1\bigl(\mathcal G'\bigr) - \BT_1\bigl(\mathcal G\bigr)=0$ if and only if $w=1$ and $d_{\mathcal G}=3p/4$.

    We now show the second part. We let $\mathcal G$ be a filtration with representation
    \[
    (1,1),(j,k),(l,m),\ldots,(p-1,p-1) \text{ with $j\leq 3p/4-1$.}
    \]
    We transform $\mathcal G$ into a filtration $\mathcal G'$ such that its start is different from $\mathcal G$'s, but from the tuple $(l,m)$ onwards, $\mathcal G$ and $\mathcal G'$ are the same. We again consider two cases.
    \begin{itemize}
        \item Case $1$: $l>j+1$. In this case, $\mathcal G'$ is represented by
        $(1,1),(j+1,k),(l,m),\ldots,(p-1,p-1).$
        \item Case $2$: $l=j+1$. In this case, $\mathcal G'$ is represented by
        $(1,1),(j+1=l,m),\ldots,(p-1,p-1).$
    \end{itemize}
    In both cases, as we show in more detail in Appendix \ref{app:lies}, we have
    \begin{equation}\label{eq:apptoprove2}
        \BT_1\bigl(\mathcal G'\bigr) - \BT_1\bigl(\mathcal G\bigr) = (k-1)\bigl(j(p-j-1)-\frac k6(3p-2k-2)\bigr)\overset{(i)}{\geq} 2\bigl((j-k+1)(j-k)\bigr)\overset{(ii)}{\geq}0,
    \end{equation}
    where we use in $(i)$ that $j\leq 3p/4-1$, and in $(ii)$ that $j\geq k$. If $n\equiv 2,4,6\mod 8$, then $j\leq 3p/4-1$ implies that $j\leq \lfloor 3p/4-1\rfloor = \lfloor 3p/4-5/4\rfloor = j_n-1$. If $n\equiv 0 \mod 8$, then $j\leq 3p/4-1 = j_n$. Only in this case, and if $j=k=3p/4-1$, then the inequality is sharp.

    Hence, if we start with an arbitrary filtration $\mathcal G$, then we can transform it step by step into a more optimal filtration $\mathcal G^1$ with $d_{\mathcal G^{1}}\leq j_n$, by the first part. Then, by the second part, we can transform $\mathcal G^1$ step by step into a more optimal filtration $\mathcal G^2$ that starts with $(1,1),(j_n,j_n)$, with alternation depth $d_{\mathcal G^{2}}= j_n$. Hence, $\mathcal G^2$ is the optimal filtration as described in the statement for $n\equiv 2,4,6\mod 8$.

    In the case that $n\equiv 0 \mod 8$, we saw in the proof of case $1$ and case $2$, that the filtration $\mathcal G^3$ starting with $(1,1),\bigl(j_{n}+1, j_n+1\bigr)$ and having alternation depth $d_{G^3}=j_n+1$ satisfies $\BT_1(\mathcal G^3) = \BT_1(\mathcal G^2)$, hence both $\mathcal G^2$ and $\mathcal G^3$ are optimal.
\end{proof}

\section{Discussion}
\label{sec:discussion}
In this paper, we have answered several key questions in the context of topological data analysis. However, many questions remain open, and we propose the following conjectures.

\begin{conjecture}
    If $\mathcal{G}$ is a filtration on $n$ vertices, then the number of intervals in the barcode of $\mathcal{G}$ in homology degree $k$ satisfies $
    |\BC_k(\mathcal{G})| \leq \BF_k(\mathcal{T}_{n,k+1}).$
\end{conjecture}

\begin{conjecture}
    The extremal filtrations described in \cref{sec:liesmagic} achieve the maximal total persistence over any edgewise filtration of flag complexes in homology degree 1.
\end{conjecture}
Proving the latter conjecture likely requires new ideas, as no filtration $\mathcal{G}$ can maximize $\BF_1(G_i)$ for all $i$. In fact, extremal values need not be realized by graphs containing Turán graphs as spanning subgraphs, as illustrated by the following example.

\begin{example}
    Let $n = 8$ and $m = 16 + 5$. Suppose $G$ contains the Turán graph $\mathcal{T}_{8,2}$. By \cref{prop:compbip}, we seek to add 5 edges to the Turán graph to maximize the product $(d_1-1)(d_2-1)$. It is easy to see that this product can be at most 1. However, by partitioning the vertices as $V = V_1 \cup V_2$, with $|V_1| = 3$ and $|V_2| = 5$, and adding all edges between $V_1$ and $V_2$, we need to add 6 more edges. Adding these as a $K_4$ in the larger partition gives $d_1 = 2$ and $d_2 = 1$; see Figure \ref{fig:ex44-35}. The value $2$ is extremal; see \cref{fig:code}.
\end{example}
\begin{figure}
    \centering
    \includegraphics[scale=0.75]{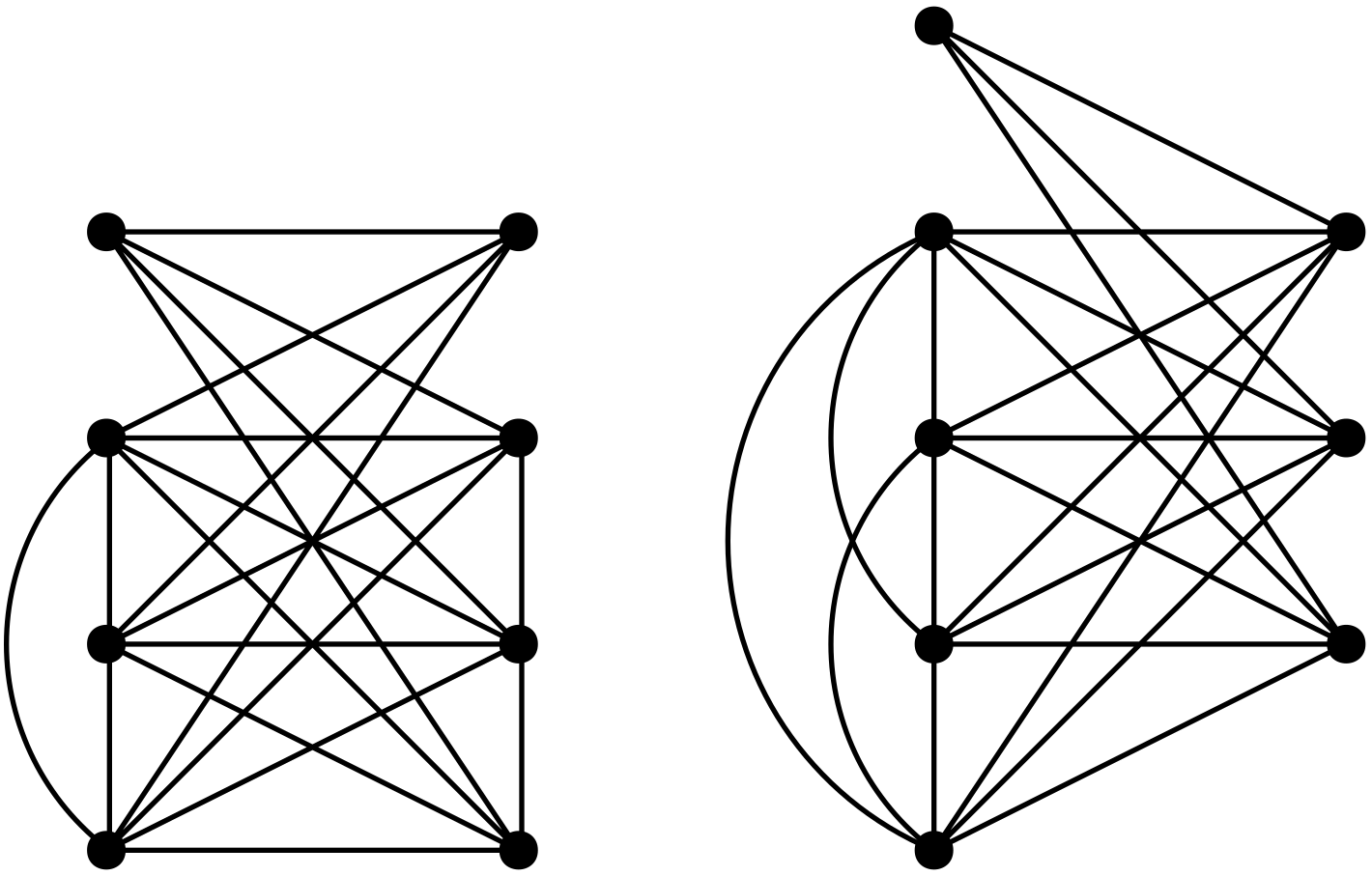}
    \caption{Two graphs $G_l$ (left) and $G_r$ (right) such that $\BF_1(G_l)=1$ and $\BF_1(G_r)=2$.}
    \label{fig:ex44-35}
\end{figure}

Finding extremal values of $\mathbb{Z}$-linear functions defined on simplicial complexes with \emph{precisely} $n$ vertices and $m$ edges presents a challenging and interesting direction for future research.

\begin{conjecture}
    Let $\mathbf{G}(n,m)$ denote the collection of graphs on $n$ vertices and $m$ edges. If $G \in \mathbf{G}(n,m)$ and $\BF_1(G) = \max_{H \in \mathbf{G}(n,m)} \BF_1(H)$, then $G$ contains a complete bipartite spanning subgraph.
\end{conjecture}

\begin{figure}
\centering
\begin{subfigure}{0.5\textwidth}
  \centering
  \includegraphics[width=\linewidth]{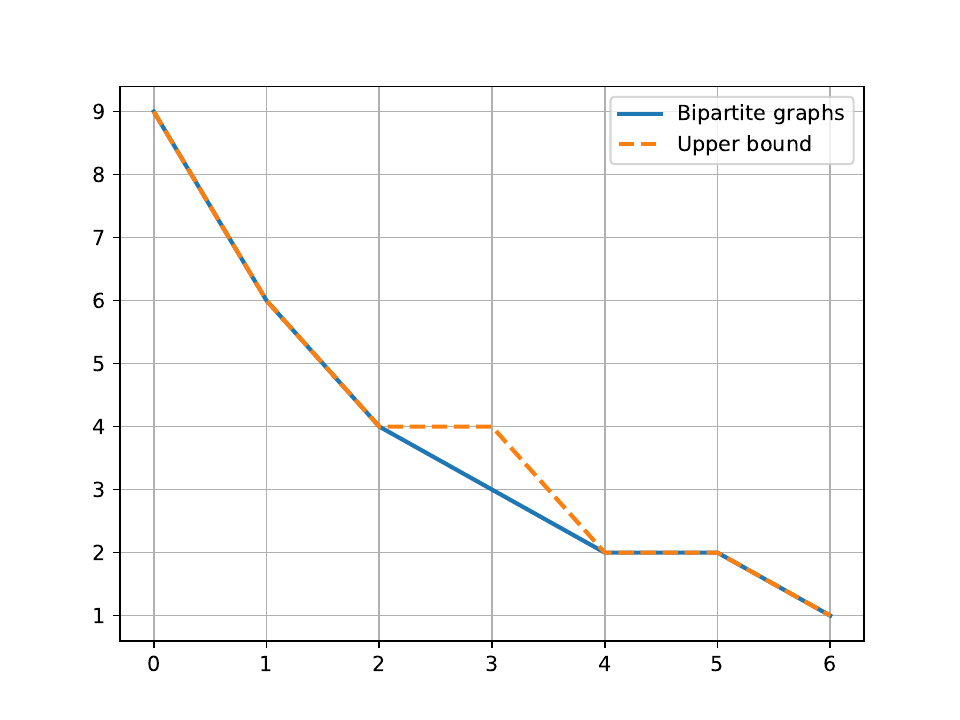}
  \label{fig:sub1}
\end{subfigure}%
\begin{subfigure}{0.5\textwidth}
  \centering
  \includegraphics[width=\linewidth]{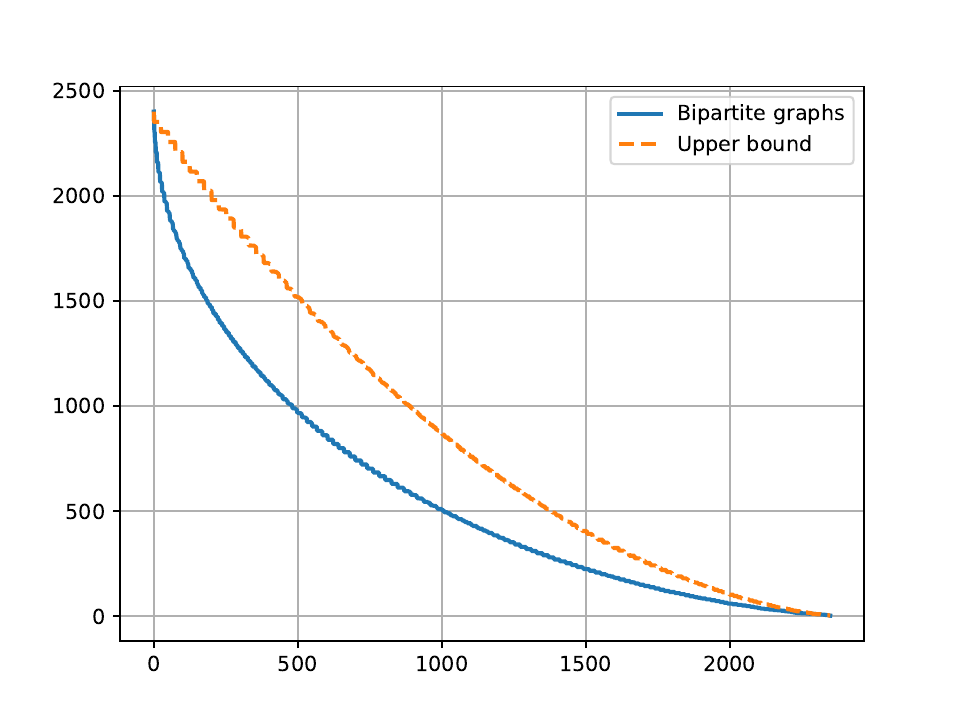}
  \label{fig:sub2}
\end{subfigure}
\caption{The vertical axis is $\BF_1$, and the value $k$ along the $x$-axis represents $(n/2)^2+k$ edges. Here $n$ is the number of vertices in the underlying graph; $n=8$ (left) and $n=100$ (right). The solid curves give the optimal value of $\BF_1(G)$ for any graph $G$ containing a complete bipartite spanning subgraph, and the dashed line is an upper bound derived from the proof of \cref{thm:turanopt} in conjuction with \cref{lem:removingedges} (details omitted). }
\label{fig:code}
\end{figure}

\bibliographystyle{plain}

\bibliography{Betti}
\appendix
\section{Flag complexes and the Vietoris--Rips complex}
\label{sec:appVR}
\begin{lemma}
\label{lem:flagrips}
    Let $\mathcal{G} = \{G_i\}_{i=1}^{m}$ be an edgewise filtration of graphs on $n$ vertices. Then, there exists a metric $d$ on the vertices of $G$, such that if the $m$ smallest pairwise distances of points in $P$ are denoted
    \[ d_1 < d_2 < \ldots < d_m,\]
    then $G_i = {\rm VR}_{d_i}(V(G))$ for all $1\leq i\leq m$. 
\end{lemma}
\begin{proof}
    Let $I({p_i, p_j})$ denote the index at which the edge $\{p_i, p_j\}$ is added to $\mathcal{G}$. Let $d(p_i, p_j) = 2 - 1/I(\{p_i, p_j\}).$ Then $d$ is a metric and the distance between $p_i$ and $p_j$ is the $I(\{p_i, p_j\}$-th shortest distance.
\end{proof}

\section{Proofs}

\subsection{Proofs from Section \ref{sec:background}}
\label{subsec:appeelementary}
\begin{lemma}[\cref{lem:MV-lk}]
Let $K$ be a simplicial complex and $v\in K$ a vertex. Then, for all $k\geq 1$,\[ \beta_k(K) \leq \beta_k(K-v) + \beta_{k-1}(\lk_K(v)).\]
\end{lemma}
\begin{proof}
Let $\Delta = \{\tau\in K : v\not\in \tau\}$ and $\Delta' = \stt_K(v)$. Then $\Delta\cup \Delta' = K$ and $\Delta \cap \Delta' = \lk_K(v)$. Observe that $\tilde{H}_k(\Delta') = 0$ as $\Delta'$ is a cone as every simplex contains $v$. Thus, Mayer-Vietoris yields the following long-exact sequence of vector spaces and linear maps
\[ \cdots\to \tilde{H}_{k+1}(\lk_K(v)) \to \tilde{H}_{k+1}(\Delta)\oplus 0 \xrightarrow{g} \tilde{H}_{k+1}(K) \xrightarrow{f} \tilde{H}_k(\lk_K(v)) \to \cdots.\]
Thus,
\[\beta_{k+1}(K) = \dim \ker f + \dim \im f  = \dim \im g + \dim \im f \leq \beta_{k+1}(\Delta) + \beta_k(\lk_k(v)).\]
\end{proof}
\begin{lemma}[\cref{lem:disjointunion}]
    For any two graphs $G$ and $H$, and $k\geq -1$, 
    \[ \beta_k(\Ind(G\sqcup H)) = \sum_{i,j\geq -1; i+j=k-1} \beta_i(\Ind(G))\beta_j(\Ind(H)),\]
    where $\sqcup$ denotes disjoint union.
\end{lemma}
\begin{proof}
    Observe that $\Ind(G\sqcup H) = \Ind(G)*\Ind(H)$ where $*$ denotes the join of simplicial complexes. The result is now immediate from Künneth for reduced homology,
    \[ \beta_k(K*L) = \sum_{i,j\geq -1; i+j=k-1} \beta_i(K)\beta_j(L),\]
\end{proof}

\begin{proposition}[\cref{prop:turanbetti}]
    For all integers $n\geq 1$ and $k\geq 1$, we let $n'$ be the smallest positive integer such that $n'\equiv n \bmod k$. We have
    \[\BF_i(\TG_{n,k}) = \begin{cases}
        (\lceil n/k\rceil-1)^{(n')}\cdot (\lfloor n/k\rfloor-1)^{k - n'} & \text{if $i=k-1$},\\
        0 & \text{otherwise.}
    \end{cases}\]
In particular, if $n$ is a multiple of $k$, then 
$\beta_{k-1} = (n/k-1)^k.$
\label{prop:turanbettiproof}
\end{proposition}
\begin{proof}
    We use the notation from Definition \ref{def:turan}.
    First, observe that 
     \[\beta_j(\Ind(K_{n_i})) = \begin{cases}
        n_i -1 & \text{if $j=0$}\\
        0 & \text{otherwise.}
        \end{cases}
        \]
We now work by induction on $k$. For $k=1$, we get $X(\mathcal{T}_{n,1}) = \Ind(K_n)$, and thus $\beta_0(\mathcal{T}_{n,1}) = n-1$ and $\beta_j$ is trivial otherwise. This covers the base case. Now, assume that the statement holds for all $j<k$. We get that, 
\[X(\mathcal{T}_{n,k}) = \Ind\bigl(\overline{\mathcal{T}_{n,k}}\bigr) = \Ind\left(\left(\bigsqcup_{i=1}^{k-1} K_{n_i}\right) \sqcup K_{n_k}\right).\]

and by \cref{lem:disjointunion}, 
\begin{align*}
    \beta_j(\mathcal{T}_{n,k}) &= \sum_{a,b\geq -1; a+b=j-1} \beta_a(\Ind(\sqcup_{i=1}^{k-1}n_i))\beta_b(\Ind(K_{n_k})) = \beta_{j-1}(\Ind(\sqcup_{i=1}^{k-1}n_i))(n_k-1)\\
    &= \prod_{i=1}^k (n_i-1)
    \end{align*}
where the last equality follows from the induction hypothesis. The result now follows from the description of the $n_i$'s in the  graph.\qedhere
\end{proof}

\subsection{Proof of Lemma \ref{lem:maxprod}}
\label{sec:maxprodproof}

\begin{lemma}[\cref{lem:maxprod}]
    Let $S$ be a positive integer, and $x_i\geq 1$ be integers satisfying $\sum_{i=1}^n x_i=S$. Then, 
 $\prod_{i=1}^n (x_i-1)\leq \prod_{i=1}^n (y_i-1)$ where
    \begin{equation}\label{eq:prodappendix}
    y_i = 
    \begin{cases}
        \lceil S/n\rceil& \text{ if } 1\leq i\leq S\bmod n\\
        \lfloor S/n\rfloor & \text{otherwise.}
    \end{cases}
    \end{equation}
\end{lemma}

\begin{proof}
Let \[M=\max\left\{\prod_{i=1}^n x_i : (x_1, \ldots, x_n) \in \mathbb{N}^n, \sum_{i=1}^n x_i = S\right\} .\]
Since we are optimizing over a finite domain, this $M$ exists. Observe that if $(x_1, \ldots, x_n)$ attains this maximum, then $|x_i-x_j|\leq 1$ for all $i,j$. If not, say, $x_i>x_j+1$, then replacing $x_i$ with $x_i-1$ and $x_j$ with $x_j+1$ would yield a larger product as $(x_i-2)x_j = (x_i-1)(x_j-1) + (x_i-x_j-1)$. Hence, $(x_1,\ldots,x_n)$ should be some permutation of $(y_1,\ldots,y_n)$ as in \eqref{eq:prodappendix}.
\end{proof}
\subsection{Proof of Lemma \ref{lem:He}}
\label{sec:proofHe}
\begin{lemma}[\cref{lem:He}]
For $e$ edges, let $m$ be maximal such that $\mathcal{T}_{m,k+1}$ is a subgraph of $H_e$. Then, \[\BF_k\bigl(H_e\bigr) = \BF_k\bigl(\mathcal{T}_{m,k+1}\bigr) + \BF_{k-1}\bigl(\mathcal{T}_{e-e_m^{k+1}, k}\bigr).\]
In particular, for $e_m^{k+1}\leq e_1 < e_2 < e_{m+1}^{k+1}$,
we have that \[\BF_k(H_{e_1+1}) - \BF_k(H_{e_1})\leq \BF_k(H_{e_2+1}) - \BF_k(H_{e_2}).\]
\end{lemma}
\begin{proof}
    Write $K=X(H_e)$, and let $v$ be a vertex of minimal degree in $K$. The key observation is that $K-v = X(\mathcal{T}_{m,k+1})$  and that $\lk_K(v) = X(\mathcal{T}_{e-e_m^{k+1}, k})$. From Mayer--Vietoris gives (see proof of \cref{lem:MV-lk}), we have the following exact sequence
    \[ \cdots\to \tilde{H}_{k}(\lk_K(v)) \to \tilde{H}_{k}(K-v)\xrightarrow{g} \tilde{H}_{k}(K) \xrightarrow{f} \tilde{H}_{k-1}(\lk_K(v)) \xrightarrow{h} \tilde{H}_{{k-1}}(K-v) \to\cdots.\]
    From \cref{prop:turanbetti}, we have $\beta_{k-1}(K-v)=0$ and $\beta_k(\lk_K(v))=0$. It follows that 
    \[\beta_k(K) = \beta_k(K-v) + \beta_{k-1}(\lk_K(v))\]
    and so
    \[\BF_k\bigl(H_e\bigr) = \BF_k\bigl(\mathcal{T}_{m,k+1}\bigr) + \BF_{k-1}\bigl(\mathcal{T}_{e-e_m^{k+1}, k}\bigr).\]

To see second part, it suffices to show that 
\[ \BF_{k-1}(\mathcal{T}_{a,k}) -  \BF_{k-1}(\mathcal{T}_{a-1,k}) \leq \BF_{k-1}(\mathcal{T}_{b,k}) -  \BF_{k-1}(\mathcal{T}_{b-1,k}), \]
for $a<b$. But this is immediate from the description of the homology of the Turan graphs (\cref{prop:turanbetti}). Specifically, $\BF_{k-1}(\mathcal{T}_{a,k})$ is the product $\prod_{i=1}^k (v_i-1)$ where $v_i$ is the number of vertices in the $i$-th partition of $\mathcal{T}_{a,k}$. Assume $\mathcal{T}_{a,k}$ is obtained from $\mathcal{T}_{a-1,k}$ by adding a vertex in partition $j$. Then, 
\[
\BF_{k-1}(\mathcal{T}_{a,k}) -  \BF_{k-1}(\mathcal{T}_{a-1,k}) =  \prod_{i=1, i\neq j}^k (v_i-1).\]
Since the number of vertices in each column of $\mathcal{T}_{b,k}$ is at least $v_i$, the result follows. 
\end{proof}

\subsection{Details for the proof of Theorem \ref{thm:lies}}
\label{app:lies}
Let $n\geq4$ be even, write $p\coloneqq n/2$, and let $e_n\coloneqq |E(\mathcal T_{n,2})|$. Consider an edgewise filtration $\mathcal G=\bigl\{G_c\bigr\}_{c=e_n+1}^{\binom{n-1}2+1}$ of $G_{\binom{n-1}2+1}$ (cf. \eqref{eq:grensgraaf}), starting from the Turán graph with one added edge, that satisfies the properties of Lemma \ref{lem:completegraphs}, that all optimal filtrations satisfy.

Any graph $G$ that we consider contains the Turán graph $\mathcal T_{n,2}$ as a subgraph. Let
$V_1$ and $V_2$ be the partition classes of this Turán graph $\mathcal T_{n,2}$ that is a subgraph of $G$. Throughout the subsection, we write $V_j$ instead of $G|_{V_j}$.

Let $i=1,\ldots,p-2$. When in $V_j$ we construct $K_{i+1}$ out of $K_{i}$, we add $i$ edges. While adding each of those edges, the product $(d_1-1)(d_2-1)$ from \eqref{eq:productcomponents} remains constant. Moreover, in this case, $d_j-1=p-i-1$, as $V_j$ contains $p-(i+1)$ isolated vertices and one component with $i+1$ vertices, containing $K_{i}$ as a subgraph. Thus,
\begin{equation}\label{eq:appeven}
    \sum_{c=e_n+1}^{\binom{n-1}2+1}\BF_1(G_c) = \sum_{i=1}^{p-2}a_i \cdot i\cdot (p-i-1) + \sum_{i=1}^{p-2}b_i \cdot i \cdot (p-i-1),
\end{equation}
where $a_i$ and $b_i$ depend on the filtration (but we suppress this in our notation): the first sum (with the $a_i$) corresponds to adding edges to $V_1$, and the second sum (with the $b_i$) corresponds to adding edges to $V_2$. We shall use this characterization throughout this subsection.

We now illustrate how to compute $a_i$ and $b_i$ from \eqref{eq:appeven} with two examples.
\begin{example}
    For the filtration from Figure \ref{fig:lb}, we have $n=10$, so $p=5$. We first add an edge to $V_1$ to form a $K_2$, and at that moment, $V_2$ consists of $p$ components. Hence, $a_1=p-1=4$. Then, we add an edge to $V_2$ (to form a $K_2$), and at that moment, $V_1$ consists of $p-1=4$ components, so $b_1=3$. Then, we form a $K_3$ in $V_1$, and the corresponding coefficient $a_2=p-2=3$. We continue like this to see that $b_2=2$, $a_3=2$ and $b_3=1$.

    Furthermore, for the filtration from Figure \ref{fig:exrep}, we find the following coefficients in the sums from \eqref{eq:appeven}:
    \[
    a_1= a_2 = 3, a_3 = 1, b_1 = 4, b_2 =b_3=2.
    \]
\end{example}
We have the following lemma, which proves our assumption from Remark \ref{rmk:leftahead}.
\begin{lemma}\label{lem:leftalwaysbigger}
Let $n\geq 4$ be even.
    Let $\mathcal G=\bigl\{G_c\bigr\}_{c=e_n+1}^{\binom{n-1}2+1}$ be an edgewise filtration of $G_{\binom{n-1}2+1}$ (cf. \eqref{eq:grensgraaf}) starting from the Turán graph with one added edge and let $(l_1,r_1),(l_2,r_2),\ldots, (l_{p-1},r_{p-1})$ be a representation as in \eqref{eq:reducednot}. We may assume that $l_i\geq r_i$ for all $i=1,\ldots,p-1$.
\end{lemma}
\begin{proof}
    We may assume that the representation $(l_1,r_1),(l_2,r_2),\ldots, (l_e,r_e)$ of $\mathcal G$ is such that $l_i+r_i=l_{i-1}+r_{i-1}+1$ for all $i=2,\ldots,d$.
    Whenever $i_1$ and $i_2$ are such that
    \[ i_2\geq i_1, l_{i_1-1}=r_{i_1-1} \text{ and } l_{i_2+1}=r_{i_2+1}, \text{ but } r_i>l_i \text{ for } i=i_1,\ldots,i_2,\]
    then replacing $(l_{i_1},r_{i_1}),\ldots,(l_{i_2},r_{i_2})$ with $(r_{i_1},l_{i_1}),\ldots,(r_{i_2},l_{i_2})$ produces a new filtration $\mathcal G'$ 
    fiberwise isomorphic to $\mathcal G$, preserving total persistence.
\end{proof}
When $n$ is odd, one can prove an analogous lemma. In this case, we can write $p\coloneqq \lfloor n/2\rfloor$, which means that $|V_1|=p+1$ and $|V_2|=p$. Hence, the equation from \eqref{eq:appeven} becomes
\begin{equation}\label{eq:appodd}
    \sum_{c=e_n+1}^{\binom{n-1}2+1}\BF_1(G_c) = \sum_{i=1}^{p-1}a_i \cdot i\cdot (p-i) + \sum_{i=1}^{p-2}b_i \cdot i \cdot (p-i-1).
\end{equation}
We have the following lemma for odd $n$.
\begin{lemma}\label{lem:leftalwaysbiggerodd}
    Let $n\geq4$ be odd, and let $p\coloneqq \lfloor n/2\rfloor$.
    Let $\mathcal G=\bigl\{G_c\bigr\}_{c=e_n+1}^{\binom{n-1}2+1}$ be an edgewise filtration of $G_{\binom{n-1}2+1}$ (cf. \eqref{eq:grensgraaf}) starting from the Turán graph with one added edge and let $(l_1,r_1),(l_2,r_2),\ldots, (l_{p-1},r_{p-1})$ represent the filtration as in \eqref{eq:reducednot}. Then, we may assume that $l_i\geq r_i$ for all $i=1,\ldots,p-1$.
\end{lemma}
\begin{proof}
    Suppose that for some part of the filtration (where we use the same notation as in the proof of Lemma \ref{lem:leftalwaysbigger}), say from $(l_{i_1},r_{i_1})$ to $(l_{i_2},r_{i_2})$, we have that $r_i>l_i$ for all $i=i_1,\ldots,i_2$. Then, we can make a filtration $\mathcal G'$ with 
    $(l'_{i_1},r'_{i_1}),\ldots,(l'_{i_2},r'_{i_2})=(r_{i_1},l_{i_1}),\ldots,(r_{i_2},l_{i_2})$. Then, we see that, for $\mathcal G$, the corresponding part of the sums from \eqref{eq:appodd} is equal to
    \begin{align*}
        \sum_{i=i_1-1}^{i_2-1}a_i \cdot i\cdot (p-i) + \sum_{i=i_1-1}^{i_2-1}b_i \cdot i \cdot (p-i-1),
    \end{align*}
    and for $\mathcal G'$, it then becomes
    \begin{align*}
        &\sum_{i=i_1-1}^{i_2-1}(a_i+1) \cdot i\cdot (p-i-1) + \sum_{i=i_1-1}^{i_2-1}(b_i-1) \cdot i \cdot (p-i) \\
        &= \sum_{i=i_1-1}^{i_2-1}i\cdot \bigl(a_i(p-i) - a_i + p-i-1\bigr) + \sum_{i=i_1-1}^{i_2-1}i\cdot \bigl( b_i(p-i-1)-(p-i-1)+b_i-1  \bigr)\\
        &= \sum_{i=i_1-1}^{i_2-1}(a_i+1) \cdot i\cdot (p-i-1) + \sum_{i=i_1-1}^{i_2-1}(b_i-1) \cdot i \cdot (p-i)\\
        &+ \sum_{i=i_1-1}^{i_2-1}i\bigl(-a_i+(p-i-1)-(p-i-1)+b_i-1\bigr)\\
        &= \sum_{i=i_1-1}^{i_2-1}a_i \cdot i\cdot (p-i) + \sum_{i=i_1-1}^{i_2-1}b_i \cdot i \cdot (p-i-1) + \sum_{i=i_1-1}^{i_2-1}i\bigl(b_i-1-a_i\bigr).
    \end{align*}
    In particular, the difference between this part of the sums for $\mathcal G'$ and $\mathcal G$ is equal to
    \[
    \sum_{i=i_1-1}^{i_2-1}i\bigl(b_i-1-a_i\bigr) \geq0,
    \]
    because by the assumption, for $i=i_1,\ldots,i_2$, the complete graph in $V_2$ is not smaller than the one in $V_1$, which means that $a_i\leq b_i-1$ for all $i=i_1,\ldots,i_2$ (because, the bigger the complete graph in $V_2$ when making $K_i$ in $V_1$, the smaller the coefficient $a_i$). The $-1$ comes from the fact that $V_1$ has one vertex more than $V_2$, which means that the number of components in $V_1$ is one bigger than the number of components in $V_2$ if both components contain a $K_i$ of the same size.
\end{proof}
Returning to the even case, we show Equations \eqref{eq:apptoprove1} and \eqref{eq:apptoprove2} in Lemmas \ref{lem:eq7} and \ref{lem:eq8}, respectively.
\begin{lemma}\label{lem:eq7}
    Let $n\geq4$ be even and let $j_n\coloneqq \lfloor (3n-2)/8\rfloor$.
    Let $\mathcal G$ be a filtration with alternation depth $d_{\mathcal G}>j_n$. Its representation, substituting $n$ by $2p$, is of the form
    \[
    (1,1),\ldots,(t,u),(v,w),\bigl(d_{\mathcal G},d_{\mathcal G}\bigr),\bigl(d_{\mathcal G}+1,d_{\mathcal G}+1\bigr),\ldots,\bigl(p-1,p-1\bigr) \text{ with } \bigl(v,w\bigr)\neq \bigl(d_{\mathcal G}-1,d_{\mathcal G}-1\bigr).
    \]
    Since we assumed that $v\geq w$, we must have $w<d_{\mathcal G}-1$. We transform $\mathcal G$ into a filtration $\mathcal G'$ which is the same as $\mathcal G$, except that we add an extra alternation step. We consider two cases.
    \begin{itemize}
        \item Case $1$: $v<d_{\mathcal G}-1$. In this case, $\mathcal G'$ is represented by
        \[
        (1,1),\ldots (t,u),(v,w),\bigl(d_{\mathcal G}-1,d_{\mathcal G}-1\bigr), \bigl(d_{\mathcal G},d_{\mathcal G}\bigr),\bigl(d_{\mathcal G}+1,d_{\mathcal G}+1\bigr),\ldots,\bigl(p-1,p-1\bigr).
        \]
        \item Case $2$: $v=d_{\mathcal G}-1$. In this case, $\mathcal G'$ is represented by
        \[
        (1,1),\ldots (t,u),\bigl(v=d_{\mathcal G}-1,d_{\mathcal G}-1\bigr), \bigl(d_{\mathcal G},d_{\mathcal G}\bigr),\bigl(d_{\mathcal G}+1,d_{\mathcal G}+1\bigr),\ldots,\bigl(p-1,p-1\bigr).
        \]
    \end{itemize}
    Then, we have that
    \begin{equation}\label{eq:alternationdepth}
        \BT_1\bigl(\mathcal G'\bigr) - \BT_1\bigl(\mathcal G\bigr) = \frac16\bigl(d_{\mathcal G}-1-w\bigr)\bigl(d_{\mathcal G}-w\bigr)\bigl(4d_{\mathcal G}+2w-3p-2\bigr).
    \end{equation}
\end{lemma}
\begin{proof}
    We first prove \eqref{eq:alternationdepth} for Case 1. In this case, we have
    \begin{align*}
        \sum_{c=e_n+1}^{\binom{n-1}2+1}\BF_1(G_c) = &\sum_{i=1}^{v-1}a_i\cdot i\cdot (p-i-1)+\sum_{i=1}^{w-1}b_i\cdot i\cdot (p-i-1) \\
        + &\sum_{i=v}^{d_{\mathcal G}-1}(p-w)\cdot i\cdot (p-i-1) + \sum_{i=w}^{d_{\mathcal G}-1}(p-d_{\mathcal G})\cdot i\cdot (p-i-1)\\
        + &\sum_{i=d_{\mathcal G}}^{p-2}i\cdot (p-i-1)\cdot (2p-2i-1), \text{ and }\\
        \sum_{c=e_n+1}^{\binom{n-1}2+1}\BF_1(G'_c) = &\sum_{i=1}^{v-1}a_i\cdot i\cdot (p-i-1)+\sum_{i=1}^{w-1}b_i\cdot i\cdot (p-i-1) \\
        + &\sum_{i=v}^{d_{\mathcal G}-2}(p-w)\cdot i\cdot (p-i-1) + \sum_{i=w}^{d_{\mathcal G}-2}(p-d_{\mathcal G}+1)\cdot i\cdot (p-i-1)\\
        + &\sum_{i=d_{\mathcal G}-1}^{p-2}i\cdot (p-i-1)\cdot (2p-2i-1).
    \end{align*}
    We see that
    \begin{align*}
        &\sum_{c=e_n+1}^{\binom{n-1}2+1}\BF_1(G'_c) - \sum_{c=e_n+1}^{\binom{n-1}2+1}\BF_1(G_c)\\
        &= (d_{\mathcal G}-1)(p-d_{\mathcal G})(-p+w-p+d_{\mathcal G}+2p-2(d_{\mathcal G}-1)-1) + \sum_{i=w}^{d_{\mathcal G}-2} i\cdot (p-i-1)\\
        &= (d_{\mathcal G}-1)(p-d_{\mathcal G})(w-d_{\mathcal G}+1) + \frac16\bigl(w-d_{\mathcal G}+1\bigr)\bigl(6p-2w-3pw+2w^2-4d_{\mathcal G}-3pd_{\mathcal G}+2wd_{\mathcal G}+2d_{\mathcal G}^2\bigr)\\
        &= \frac16\bigl(d_{\mathcal G}-1-w\bigr)\bigl(d_{\mathcal G}-w\bigr)\bigl(4d_{\mathcal G}+2w-3p-2\bigr).
    \end{align*}
    Furthermore, in case $2$, we see that
    \begin{align*}
        \sum_{c=e_n+1}^{\binom{n-1}2+1}\BF_1(G_c) &=\sum_{i=1}^{t-1}a_i\cdot i\cdot (p-i-1)+\sum_{i=1}^{u-1}b_i\cdot i\cdot (p-i-1) \\
        &+ \sum_{i=t}^{d_{\mathcal G}-2}(p-u)\cdot i\cdot (p-i-1)+\sum_{i=u}^{w-1}(p-d_{\mathcal G}+1)\cdot i\cdot (p-i-1) \\
        &+ (p-w)(d_{\mathcal G}-1)(p-d_{\mathcal G}) + \sum_{i=w}^{d_{\mathcal G}-1}(p-d_{\mathcal G})\cdot i\cdot (p-i-1)\\
        &+ \sum_{i=d_{\mathcal G}}^{p-2}i\cdot (p-i-1)\cdot (2p-2i-1),
    \end{align*}
    and
    \begin{align*}
        \sum_{c=e_n+1}^{\binom{n-1}2+1}\BF_1(G'_c) &= \sum_{i=1}^{t-1}a_i\cdot i\cdot (p-i-1)+\sum_{i=1}^{u-1}b_i\cdot i\cdot (p-i-1) \\
        &+ \sum_{i=t}^{d_{\mathcal G}-2}(p-u)\cdot i\cdot (p-i-1) + \sum_{i=u}^{d_{\mathcal G}-2}(p-d_{\mathcal G}+1)\cdot i\cdot (p-i-1)\\
        &+ \sum_{i=d_{\mathcal G}-1}^{p-2}i\cdot (p-i-1)\cdot (2p-2i-1).
    \end{align*}
    Also in this case, we see that
    \begin{align*}
        &\sum_{c=e_n+1}^{\binom{n-1}2+1}\BF_1(G'_c) - \sum_{c=e_n+1}^{\binom{n-1}2+1}\BF_1(G_c)\\
        &= \sum_{i=w}^{d_{\mathcal G}-2}\bigl(p-d_{\mathcal G}+1-(p-d_{\mathcal G})\bigr)\cdot i\cdot (p-i-1) + \bigl(-p+w-p+d_{\mathcal G}+2p-2(d_{\mathcal G}-1)-1\bigr)(d_{\mathcal G}-1)(p-d_{\mathcal G}) \\
        &= (d_{\mathcal G}-1)(p-d_{\mathcal G})(w-d_{\mathcal G}+1) + \sum_{i=w}^{d_{\mathcal G}-2}i\cdot (p-i-1)\\
        &= \frac16\bigl(d_{\mathcal G}-1-w\bigr)\bigl(d_{\mathcal G}-w\bigr)\bigl(4d_{\mathcal G}+2w-3p-2\bigr),
    \end{align*}
    analogously to Case $1$.
\end{proof}
Finally, we prove Equation \eqref{eq:apptoprove2} from Theorem \ref{thm:lies}.
\begin{lemma}\label{lem:eq8}
    Let $n\geq4$ be even.
    We let $\mathcal G$ be a filtration with representation
    \[
    (1,1),(j,k),(l,m),\ldots,(p-1,p-1) \text{ with $j\leq 3p/4-1$.}
    \]
    We transform $\mathcal G$ into a filtration $\mathcal G'$ such that its start is different from $\mathcal G$'s, but from the tuple $(l,m)$ onwards, $\mathcal G$ and $\mathcal G'$ are the same. We again consider two cases.
    \begin{itemize}
        \item Case $1$: $l>j+1$. In this case, $\mathcal G'$ is represented by
        $(1,1),(j+1,k),(l,m),\ldots,(p-1,p-1).$
        \item Case $2$: $l=j+1$. In this case, $\mathcal G'$ is represented by
        $(1,1),(j+1=l,m),\ldots,(p-1,p-1).$
    \end{itemize}
    In both cases, we have
    \begin{equation}\label{eq:biggerstart}
        \BT_1\bigl(\mathcal G'\bigr) - \BT_1\bigl(\mathcal G\bigr) = (k-1)\bigl(j(p-j-1)-\frac k6(3p-2k-2)\bigr).
    \end{equation}
\end{lemma}
\begin{proof}
    We first prove \eqref{eq:biggerstart} for Case $1$. In this case, we have
    \begin{align*}
        \sum_{c=e_n+1}^{\binom{n-1}2+1}\BF_1(G_c) &= \sum_{i=1}^{j-1}(p-1)\cdot i\cdot (p-i-1) + \sum_{i=1}^{k-1}(p-j)\cdot i\cdot (p-i-1)\\
        &+ \sum_{i=j}^{p-2}a_i \cdot i\cdot (p-i-1) + \sum_{i=k}^{p-2}b_i \cdot i \cdot (p-i-1), \text{ and }\\
        \sum_{c=e_n+1}^{\binom{n-1}2+1}\BF_1(G'_c) &= \sum_{i=1}^{j}(p-1)\cdot i\cdot (p-i-1) + \sum_{i=1}^{k-1}(p-j-1)\cdot i\cdot (p-i-1)\\
        &+ \sum_{i=j+1}^{p-2}a_i \cdot i\cdot (p-i-1) + \sum_{i=k}^{p-2}b_i \cdot i \cdot (p-i-1).
    \end{align*}
    Since $a_j=p-k$, we see that
    \begin{align*}
        &\sum_{c=e_n+1}^{\binom{n-1}2+1}\BF_1(G'_c) - \sum_{c=e_n+1}^{\binom{n-1}2+1}\BF_1(G_c)\\
        &= (p-1)\cdot j\cdot (p-j-1) - \sum_{i=1}^{k-1}i\cdot (p-i-1) - (p-k)\cdot j\cdot (p-j-1)\\
        &= (k-1)\cdot j\cdot (p-j-1)-(k-1)\cdot \frac k6\cdot (3p - 2k - 2).
    \end{align*}
    For Case $2$, we have
    \begin{align*}
        \sum_{c=e_n+1}^{\binom{n-1}2+1}\BF_1(G_c) &= \sum_{i=1}^{j-1}(p-1)\cdot i\cdot (p-i-1) + \sum_{i=1}^{k-1}(p-j)\cdot i\cdot (p-i-1)\\
        &+(p-k)\cdot j\cdot (p-j-1) + \sum_{i=k}^{m-1}(p-j-1)\cdot i\cdot (p-i-1)\\
        &+ \sum_{i=j+1}^{p-2}a_i \cdot i\cdot (p-i-1) + \sum_{i=m}^{p-2}b_i \cdot i \cdot (p-i-1), \text{ and }\\
        \sum_{c=e_n+1}^{\binom{n-1}2+1}\BF_1(G'_c) &= \sum_{i=1}^{j}(p-1)\cdot i\cdot (p-i-1) + \sum_{i=1}^{m-1}(p-j-1)\cdot i\cdot (p-i-1)\\
        &+ \sum_{i=j+1}^{p-2}a_i \cdot i\cdot (p-i-1) + \sum_{i=m}^{p-2}b_i \cdot i \cdot (p-i-1).
    \end{align*}
    Also in this case, we see that
    \begin{align*}
        &\sum_{c=e_n+1}^{\binom{n-1}2+1}\BF_1(G'_c) - \sum_{c=e_n+1}^{\binom{n-1}2+1}\BF_1(G_c)\\
        &= (p-1)\cdot j\cdot (p-j-1) - (p-k)\cdot j\cdot (p-j-1) - \sum_{i=1}^{k-1}i\cdot (p-i-1) \\
        &= (k-1)\cdot j\cdot (p-j-1)-(k-1)\cdot \frac k6\cdot (3p - 2k - 2).
    \end{align*}
\end{proof}
We conclude with a final remark.
\begin{remark}
    The proof for the optimal filtration for odd $n$, as described in the statement of Theorem \ref{thm:lies}, is analogous to the proof of the even case. It relies on analogous versions of Lemmas \ref{lem:eq7} and \ref{lem:eq8}, using Equation \eqref{eq:appodd} instead of \eqref{eq:appeven}.
\end{remark}
\end{document}